\documentclass[11pt]{article}
\usepackage{cite}
\usepackage{amssymb}
\usepackage{amsfonts}
\usepackage{mathrsfs}
\usepackage{graphicx}
\usepackage{amsmath}
\usepackage{color}
\usepackage{amsthm}
\usepackage{epstopdf}
\usepackage{booktabs}
\usepackage{pifont,bm}
\usepackage{tikz}
\usepackage{enumerate}
\usepackage{threeparttable}
\usepackage{diagbox}

\newtheorem{thm}{Theorem}[section]

\newtheorem{proposition}[thm]{Proposition}
\theoremstyle{definition}

\newtheorem{theorem}[thm]{Theorem}
\newtheorem{remark}[thm]{Remark}

\newtheorem{rhp}[thm]{Riemann-Hilbert problem}

\makeatletter
\def\@biblabel#1{[#1]}
\makeatother

\makeatletter \@addtoreset{equation}{section}

\begin{document}

\begin{titlepage}
\title{\bf{Long time asymptotic analysis for a nonlocal Hirota equation via the Dbar steepest descent method
\footnote{
Corresponding authors.\protect\\
\hspace*{3ex} E-mail addresses: ychen@sei.ecnu.edu.cn (Y. Chen)}
}}
\author{Jin-yan Zhu$^{a}$, Yong Chen$^{a,b,*}$\\
\small \emph{$^{a}$School of Mathematical Sciences, Shanghai Key Laboratory of PMMP} \\
\small \emph{East China Normal University, Shanghai, 200241, China} \\
\small \emph{$^{b}$College of Mathematics and Systems Science, Shandong University }\\
\small \emph{of Science and Technology, Qingdao, 266590, China} \\
\date{}}
\thispagestyle{empty}
\end{titlepage}
\maketitle

\vspace{-0.5cm}
\begin{center}
\rule{15cm}{1pt}\vspace{0.3cm}

\parbox{15cm}{\small
{\bf Abstract}\\
\hspace{0.5cm} In this paper, we mainly focus on the Cauchy problem of an integrable nonlocal Hirota equation with initial value in weighted  Sobolev space. Through the spectral analysis of Lax pairs, we successfully transform the Cauchy problem of the nonlocal Hirota equation into a solvable Riemann-Hilbert problem. Furthermore, in the absence of discrete spectrum, the long-time asymptotic behavior of the solution for the nonlocal Hirota equation is obtained through the Dbar steepest descent method. Different from the local Hirota equation, the leading order term on the continuous spectrum and residual error term of $q(x,t)$ are affected by the function $Im\nu(z_j)$.
}

\vspace{0.5cm}
\parbox{15cm}{\small{

\vspace{0.3cm} \emph{Key words:} Nonlocal Hirota equation; Riemann-Hilbert problem; Long-time asymptotic; Dbar steepest descent method.\\

 } }
\end{center}
\tableofcontents
\vspace{0.3cm} \rule{15cm}{1pt} \vspace{0.2cm}

\section{Introduction}
In recent years, many scholars have done a lot of research on integrable equations from various aspects\cite{KN-1978-JMP,VAE-1978-ZETF,VAK-1989-PD,YZ-2001-PLA,ZC-2021-JMP,PC-2022-JMP}. Later, researchers found that more extensive nonlocal equations also have many applications in optical lattices \cite{SN-2012-PRA}, optics\cite{AMJ-2013-PRL}, and quantum mechanics \cite{CM-2002-PRL}. Among which the most famous nonlocal nonlinear Schr\"{o}dinger(NNLS) equation
\begin{equation}\label{nls}
i u_{t}(x, t)-u_{x x}(x, t)-2\sigma Q(x,t) u(x, t)=0,~~\sigma=\pm1,~Q(x,t)=u(x, t) u^{*}(-x, t)
\end{equation}
was first discovered by Ablowitz and Musslimani in 2013\cite{AMJ-2013-PRL}. Symbol $\ast$ denotes complex conjugation. The soliton solution of the NNLS is studied by various methods, and its long-time asymptotic behavior is also investigated in different spaces with initial values\cite{AMJ-2016-N,FBF-2018-N,WM-2021-ND,RY-2019-JMP,LGZ-2021-AR}. In fact, Eq. (\ref{nls}) is $P T$ symmetric and satisfies the potential $ Q(x,t)=Q^{*}(-x,t)$. In this case, the soliton structure has a stable state under certain conditions, and the propagation process is almost lossless. Moreover, $PT$ symmetry is an important reason for the appearance of nonlocal models of the real spectrum of non Hamiltonian operators\cite{BCM-1998-PRL}, so it keeps the system unchanged under parity and time reversal transformations: $P:x\rightarrow-x,T:t\rightarrow-t$ and $i\rightarrow-i$\cite{RCE-2010-NP}. These properties are very important in soliton theory, so a large class of nonlocal equations are proposed and studied\cite{GM-2019-CNS,MWX-2020-JGP,ZZX-2018-CNS,ZH-2021-ND}.

The nonlocal Hirota equation(also call reverse space nonlocal Hirota equation)
\begin{equation}\label{nhe}
i u_{t}+\alpha\left[u_{x x}-2 \kappa {u}^{*}(-x,t) u^{2}\right]+i\beta\left[u_{x x x}-6 \kappa u {u}^{*}(-x,t) u_{x}\right]=0,~\kappa=\pm1,\alpha,\beta \in \mathbb{R}
\end{equation}
which is reduced by
$$
\begin{gathered}
u_{t}-i \alpha u_{x x}+2 i \alpha u^{2} v+\beta\left[u_{x x x}-6 u v u_{x}\right]=0, \\
v_{t}+i \alpha v_{x x}-2 i \alpha u v^{2}+\beta\left(v_{x x x}-6 u v v_{x}\right)=0
\end{gathered}
$$
 under condition $v=\kappa u^{*}(-x,t)$\cite{CJ-2019-JMP}. For this generalized coupled system of equations, there are many reduced forms. If $v=\kappa u^{*}(-x,-t)$, it is called reverse space-time nonlocal Hirota equation. If $v=\kappa u^{*}(x,-t)$ corresponds to reverse time nonlocal Hirota equation. These reduced equations are integrable, and its multi-soliton solutions are given in Ref.\cite{CJ-2019-JMP} by Hirota bilinear method. Soliton solutions of the reverse space nonlocal Hirota equation through Darboux transformation and  inverse scattering transform(IST) have also been learned\cite{LY-2022-CPAA,LNN-2021-ND}. The N-th Darboux transform was constructed and some types of exact solutions were obtained for the reverse space-time nonlocal Hirota equation\cite{XYR-2021-CPB}. In Ref.\cite{PWQ-2022-PD}, the Riemann-Hilbert method is used to study N-soliton solution of the reverse space-time nonlocal Hirota equation under non-zero boundary conditions. All of these give exact solutions without scattering, but the classical GLM equation or RH method can not solve them accurately with reflection. Therefore, it is necessary to understand the properties and structure of the solutions by studying the asymptotic behavior of the solutions of nonlinear partial differential equations.

The research on the long-time asymptotic behavior of the solutions for integrable equations can be traced back to 1974. Manakov et al. first proposed using the IST to solve the long-time asymptotic behavior of nonlinear wave equations \cite{M-1973-ZETF,AN-1973-JMP}. Later, the long-time asymptotic behavior of the solution for the fast decay initial value problem for a large number of nonlinear integrable systems was studied by using the IST\cite{ZM-1976-SPJ,Its-1981-DAN}. In 1993, Deift and Zhou, academicians of the American Academy of Sciences, invented a new nonlinear speed descent method based on the classical speed descent method, which is called Deift-Zhou speed descent method. Through a series of deformations of the original RH problem, it is approximated as a solvable model described by a parabolic cylindrical function. With the help of this special function, the expression of the long-time asymptotic behavior of the solution of the mKdV equation was strictly given\cite{DZ-1993-AM}. Once proposed, the method has been favored by scholars at home and abroad, and has been widely used in many integrable equations, such as focusing NLS equation, KdV equation, derivative nonlinear Schr\"{o}dinger equation, Degasperis-Procesi equation,Kundu-Eckhaus equation etc\cite{XD-1994-UT,KG-2009-MPAG,AD-2013-N,XFC-2013-MPAG,WGW-2019-JDE}.

Recently, McLaughlin and Miller further extended the Deift-Zhou speed descent method, and proposed the Dbar speed descent method with the help of the conclusion of the Dbar problem\cite{MM-2006-IMRN,MM-2008-IMRN}. This method has been successfully used to study the asymptotic stability of NLS multi-soliton solutions and the long-time approximation of KdV equation and NLS equation\cite{SR-2016-CMP,MR-2018-AN,P-2017-N}. The Dbar speed descent method is different from the Deift-Zhou speed descent method in analyzing the asymptotic properties of orthogonal polynomials in singular integrals on the jump line. Instead, the discontinuous part on the jump line is transformed into the form of Dbar problem. At the same time, the estimation of Cauchy integral operator in $L^{P}$ space is avoided, and the discontinuous part of Riemann-Hilbert problem(RHP) is rewritten as Dbar problem, which can be solved by integral equation. Because of its advantages, it is widely used in integrable equations, and has made great progress\cite{RJ-2018-CMP,YTL-2021-ar,CF-2022-JDE,YF-2022-AM,ZJY-2022-ar}.

In this paper, we mainly investigate the long-time asymptotic behavior of the solution for Eq.(\ref{nhe}) with weighted Sobolev initial value
 $$
u(x, 0) \in H^{1,1}(\mathbb{R}),~~H^{1,1}(\mathbb{R})=\{f \in L^{2}(\mathbb{R}):\left(1+|\cdot|^{2}\right)^{\frac{1}{2}} f\cap f^{\prime} \in L^{2}(\mathbb{R})\}
$$
by the Dbar speed descent method. It is worth mentioning that in this paper, we prove that when the initial value belongs to $H^{1,1}(\mathbb{R})$, the scattering coefficients $r(z),\tilde{r}(z)$ persist in $H^{1,1}(\mathbb{R})$ with the passage of time. And here we consider the case that there are no singular points on the continuous spectrum. Compared with the classical Hirota equation\cite{YTL-2021-ar}, the leading order term on the continuous spectrum and residual error term of the long-time asymptotic behavior of the solution for the nonlocal Hirota equation  are affected by the function $Im\nu(z_j)$.

The lay out of the paper is as follows. In section 2, we analyze the spectrum of the Eq. (\ref{nhe}) based on Lax pair. The analyticity, symmetry and asymptotic behavior of characteristic function and scattering matrix are studied. We also prove that when the initial value belongs to weighted Sobelev space, $r(z),\tilde{r}(z)$ belong to $H^{1,1}(\mathbb{R})$. In section 3, we establish associated the RHP of the nonlocal  Hirota equation(\ref{nhe}). In section 4,  by introducing $\delta(z)$ function, we change $M(x,z)$ into $M^{(1)}(x,z)$ and get a new RHP about $M^{(1)}(x,z)$, which allows the jump matrix near the phase points can be decomposed into two kinds of triangles. In section 5, The mixed $\bar{\partial}$-RHP is obtained for the nonlocal Hirota equation by defining $\mathcal{R}^{(1)}$. In section 6, we decompose the mixed $\bar{\partial}$-RHP into the model RHP of $M_{rhp}(z)$ and the pure $\bar{\partial}$ problem of $ M^{(3)}(z)$. Finally, the long-time asymptotic behavior of the nonlocal Hirota equation is obtained in section 7.

\section{The spectral analysis}

Eq.(\ref{nhe}) has Lax pair is
\begin{equation}\label{lax}
\begin{array}{l}
\Psi_{x}=U\Psi=\left(\begin{array}{cc}
-i z & u(x,t)\\
\kappa u^{*}(-x,t) & i z
\end{array}\right) \Psi, \\
 \Psi_{t}={V} \Psi=\left(\begin{array}{cc}
A &B\\
C & -A
\end{array}\right) \Psi,
\end{array}
\end{equation}
where
$$
\begin{aligned}
&A=-i \alpha\kappa  u u^{*}(-x,t)-2 i \alpha z^{2}+\beta\left(\kappa u^{*}(-x,t)  u_{x}+\kappa u u^{*}_{x}(-x,t)-4 i z^{3}-2 i z \kappa u u^{*}(-x,t)\right), \\
&B=i \alpha u_{x}+2 \alpha z u+\beta\left(2 \kappa u^{2} u^{*}(-x,t)-u_{x x}+2 i z u_{x}+4 z^{2} u\right), \\
&C=i \alpha\kappa u^{*}_{x}(-x,t)+2 \alpha\kappa z u^{*}(-x,t)+\beta\left(2 \kappa u u^{*2}(-x,t)-\kappa u^{*}_{x x}(-x,t)+2 i z\kappa u^{*}_{x}(-x,t)+4 z^{2}\kappa u^{*}(-x,t)\right).
\end{aligned}
$$
The spectral parameter $z\in \mathbb{C}$, under the rapidly decaying initial condition
$$
u(x,0)=u_0(x),~~~\lim _{|x| \rightarrow \infty}u_0(x)\rightarrow0.
$$
The Lax pair (\ref{lax}) admits the Jost solution $\Psi$:
$$
\Psi\thicksim e^{-izx\sigma_3-(2i\alpha z^2+4i\beta z^3)t\sigma_3},~~~|x|\rightarrow\infty
$$
where
$$
\sigma_{3}=\left[\begin{array}{rr}
1 & 0 \\
0 & -1
\end{array}\right].
$$
Let make the following change
\begin{equation}
\Phi=\Psi e^{izx\sigma_3+(2i\alpha z^2+4i\beta z^3)t\sigma_3},\label{pp}
\end{equation}
the modified Lax pair can be obtained
\begin{equation}
\begin{aligned}\label{ut}
&\Phi_{x}+i z\left[\sigma_{3}, \Phi\right] =Q \Phi, \\
&\Phi_{t}+(2i\alpha z^2+4i\beta z^3)\left[\sigma_{3}, \Phi\right] =V_1 \Phi,
\end{aligned}
\end{equation}
where
$$
Q=\left[\begin{array}{rr}
0 & u(x,t)\\
\kappa  u^{*}(-x,t) & 0
\end{array}\right],
$$
$$
V_1=-(i\alpha+2i\beta z) (Q^2\sigma_{3}-Q_x\sigma_{3})+\beta[Q_x,Q]\sigma_{3}+(2\alpha z+4\beta z^2) Q+2\beta Q^3-\beta Q_{xx},
$$
which can be written in full derivative form
$$
d(e^{i\left(zx+2\alpha z^2t+4\beta z^3t\right)\widehat{\sigma}_{3}} \Psi(x, t, k))=e^{i\left(zx+2\alpha z^2t+4\beta z^3t\right)\widehat{\sigma}_{3}}[Qdx+V_1dt)\Phi].
$$
The asymptotic solution $\Psi(x,t,k)$ satisfies
$$
\Phi^{\pm}(x, t, k) \rightarrow I, \quad x \rightarrow \pm \infty,
$$
the modified solution $\Phi^{\pm}(x, t, k)$ satisfy the following integral equations
\begin{equation}
\begin{aligned}
&\Phi^{-}(x, t, k)=I+\int_{- \infty}^{x} \mathrm{e}^{-i z(x-y) \sigma_{3}}Q(y, t) \Phi^{-}(y, t, k) \mathrm{e}^{i z(x-y) \sigma_{3}} \mathrm{~d} y,\\
&\Phi^{+}(x, t, k)=I-\int_{x}^{\infty} \mathrm{e}^{-i z(x-y) \sigma_{3}}Q(y, t) \Phi^{+}(y, t, k) \mathrm{e}^{i z(x-y) \sigma_{3}} \mathrm{~d} y.\label{jost}
\end{aligned}
\end{equation}

Dividing $\Phi^{\pm}$ into columns as $\Phi^{\pm}=\left(\Phi^{\pm}_{1}, \Phi^{\pm}_{2}\right)$, due to the structure of the potential $Q$, and Volterra integral equation (\ref{jost}), we have
\begin{proposition}\label{pr1}
For $u_0(x) \in H^{1,1}(\mathbb{R})$ , there exist unique eigenfunctions $\Phi^{\pm}$ which satisfy Eq.(\ref{jost}), respectively, and have the following properties:

$\bullet$  $\text{det} \Phi^{\pm}(x,t,z)=1,  x, t,z \in \mathbb{R}$;

$\bullet$  $ \Phi^{\pm}(x,t,z)\rightarrow I,  z\rightarrow\infty$;

$\bullet$ $\Phi^{-}_{1},\Phi^{+}_{2}$ are analytic in $\mathbb{C}_{+}$ and continuous in $\mathbb{C}_{+}\cup \mathbb{R}$;

 $\bullet$ $\Phi^{+}_{1},\Phi^{-}_{2}$ are analytic in $\mathbb{C}_{-}$, and continuous in $\mathbb{C}_{-}\cup \mathbb{R}$.

\end{proposition}

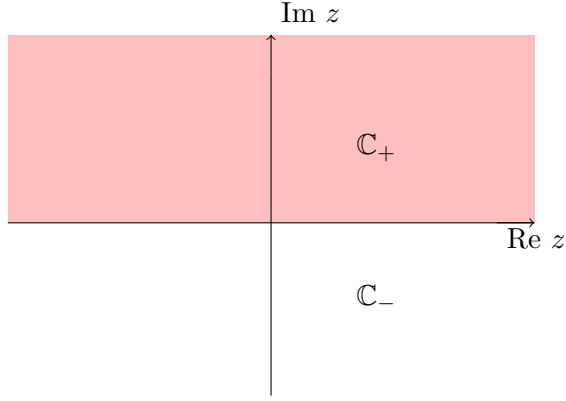
\begin{figure}
\center
\begin{tikzpicture}\usetikzlibrary{arrows}
\coordinate [label=0: Im $z$] ()at (0,2.8);
\coordinate [label=0: Re $z$] ()at (3,-0.2);
\coordinate [label=0:] ()at (2,0.1);
\coordinate [label=0:] ()at (-2.6,0.1);
\path [fill=pink] (-3.5,2.5)--(3.5,2.5) to
(3.5,0)--(-3.5,0);
\coordinate [label=0: $\mathbb{C}_{+}$] ()at (1,1);
\coordinate [label=0: $\mathbb{C}_{-}$] ()at (1,-1);
\draw[->, ] (3,0)--(3.5,0);
\draw[->,] (0,2)--(0,2.5);
\draw[thin](0,2)--(0,-2.3);
\draw[thin](3.5,0)--(-3.5,0);
\end{tikzpicture}
\caption{\small Definition of the $\mathbb{C}_{+}=\{z \mid \operatorname{Im} z>0\},~\mathbb{C}_{-}=\{z \mid \operatorname{Im} z<0\}$ }\label{1t}
\end{figure}

The partition of $\mathbb{C}_{\pm}$ is shown in Fig.\ref{1t}. Since $\Psi^{+}(x,t,z)$ and $\Psi^{-}(x,t,z)$ are solved by (\ref{lax}), the two solutions of an equation must be linearly related, and there is a constant matrix $S(z)$ independent of $x$,
\begin{equation}\label{ps}
\Psi^{-}(x,t,z)=\Psi^{+}(x,t,z)S(z),~~~z\in \mathbb{R},
\end{equation}
$S(z)$ is commonly referred to as the scattering matrix. Using Cramer's rule, we get the coefficient of the scattering matrix from (\ref{pp}) and (\ref{ps}) as
\begin{equation}\label{s1122}
\begin{array}{ll}
s_{11}(z)=\operatorname{det}\left(\Psi_{1}^{-}, \Psi_{2}^{+}\right), & s_{12}(z)=\operatorname{det}\left(\Psi_{2}^{-}, \Psi_{2}^{+}\right),\\
s_{21}(z)=\operatorname{det}\left(\Psi_{1}^{+}, \Psi_{1}^{-}\right), & s_{22}(z)=\operatorname{det}\left(\Psi_{1}^{+}, \Psi_{2}^{-}\right).
\end{array}
\end{equation}
On the basis of the relation (\ref{pp}) and properties of $\Phi^{\pm}(x,t,z)$ in Proposition (\ref{pr1}), we known that
\begin{proposition}\label{pr2}
The scattering matrix $S(z)$ and scattering data $s_{11}(z),s_{22}(z)$ satisfies:

$\bullet$  $\text{det} S(z)=1, \quad x, t \in \mathbb{R}$;

$\bullet$  $ S(z)\rightarrow I, \quad z\rightarrow\infty$;

$\bullet$  $s_{11}(z)$ is analytic in $\mathbb{C}_{+}$ and continuous in $\mathbb{C}_{+}\cup \mathbb{R} ;$

$\bullet$  $s_{22}(z)$ is analytic in $\mathbb{C}_{-}$ and continuous in $\mathbb{C}_{-}\cup \mathbb{R} ;$
\end{proposition}
By the way, we can get
$$
\begin{array}{ll}
s_{11}(z) \sim 1+\mathcal{O}\left(z^{-1}\right), & s_{12}(z) \sim \mathcal{O}\left(z^{-1}\right), \\
s_{22}(z) \sim 1+\mathcal{O}\left(z^{-1}\right), & s_{21}(z) \sim \mathcal{O}\left(z^{-1}\right).
\end{array}
$$
Next, it is very important to study the symmetry of characteristic function $\Phi(x,t,z)$ and scattering matrix $S(z)$. In essence, the symmetry of $U$ and $V$ are guided according to the symmetry of $\Phi(x,t,z)$ and $S(z)$, which is given by the following proposition:
\begin{proposition}
The $U(x, t, z)$ and $V(x, t, z)$ in the Lax pair (\ref{lax}) meet the following reduction conditions on z-plane:
$$
U(x, t, z)=-\sigma_{0} U\left(-x,t,-z^{*}\right)^{*} \sigma^{-1}_{0}, \quad V(x, t, z)=-\sigma_{0} V\left(-x,t,-z^{*}\right)^{*} \sigma^{-1}_{0},
$$
which lead to the symmetry for the $\Phi(x,t,z)$ and $S(z)$:
\begin{equation}
\Phi^{\pm}(x, t, z)=\sigma_{0} \Phi^{\mp}\left(-x,t,-z^{*}\right)^{*} \sigma^{-1}_{0}, \quad S(z)=\sigma_{0} S^{*}\left(-z^{*}\right)^{-1} \sigma^{-1}_{0}
\end{equation}
where
$$
\sigma_{0}=\left[\begin{array}{rr}
0 & -\kappa \\
1 & 0
\end{array}\right].
$$
\end{proposition}
Further, the symmetry of the scattering coefficient can be written as
\begin{equation}\label{s12z}
s_{11}(z)=s_{11}^{*}(-z^{*}), ~~~s_{12}(z)=\kappa s_{21}^{*}(-z^{*}),~~s_{22}(z)=s_{22}^{*}(-z^{*}).
\end{equation}
Next, we define the reflection coefficient
\begin{equation}\label{rz}
r(z)=\frac{s_{21}(z)}{s_{11}(z)},~~~\tilde{r}(z)=\frac{s_{12}(z)}{s_{22}(z)},
\end{equation}
looking back at Proposition (\ref{pr2}), we know that $det S=1$, then
$$
1-r(z)\tilde{r}(z)=\frac{1}{s_{11}(z)s_{22}(z)}.
$$
Based on the above analysis, we can get the following theorem about $r(k)$.
\begin{thm}\label{rh}
For any given initial value $u\in H^{1,1}(\mathbb{R})$, we have $r(k)$ and $\tilde{r}(z)$ belong to $H^{1,1}(\mathbb{R})$.
\end{thm}
\begin{proof}
Firstly, we consider the space part, that is, the initial case when $t = 0$, which can be known from the integral equation satisfied by the characteristic function
\begin{equation}\label{psz}
\begin{aligned}
&{\Phi^{-}_{1}(x,z)=\left(\begin{array}{l}
1 \\
0
\end{array}\right)+\int_{-\infty}^{x}\left(\begin{array}{c}
u(y)\Phi^{-}_{21}(y,z) \\
\kappa u^{*}(-y) e^{2 i z(x-y)}\Phi^{-}_{11}(y, z)
\end{array}\right) d y}, \\
&{\Phi_{2}^{-}(x, z)=\left(\begin{array}{c}
0 \\
1
\end{array}\right)+\int_{-\infty}^{x}\left(\begin{array}{c}
 u(y)e^{-2 i z(x-y)}\Phi^{-}_{22}(y, z) \\
\kappa u^{*}(-y)\Phi^{-}_{12}(y, z)
\end{array}\right) d y},\\
&{\Phi^{+}_{1}(x,z)=\left(\begin{array}{l}
1 \\
0
\end{array}\right)-\int^{\infty}_{x}\left(\begin{array}{c}
u(y)\Phi^{+}_{21}(x,z) \\
\kappa u^{*}(-y) e^{2 i z(x-y)}\Phi^{+}_{11}(y, z)
\end{array}\right) d y}, \\
&{\Phi_{2}^{+}(x, z)=\left(\begin{array}{c}
0 \\
1
\end{array}\right)-\int^{\infty}_{x}\left(\begin{array}{c}
 u(y)e^{-2 i z(x-y)}\Phi^{+}_{22}(y, z) \\
\kappa u^{*}(-y)\Phi^{+}_{12}(y, z)
\end{array}\right) d y}.
\end{aligned}
\end{equation}
Initially, we prove the boundedness of the characteristic function. Here, take $\Phi^{-}_{1}(x,z)$ as an example, others can be obtained similarly.
First, for $k \in \mathbb{R},f\in L^{\infty}(-\infty, 0]$, we define an operator mapping
$$
\Gamma[f](x)=\int_{-\infty}^{x}\left(\begin{array}{rr}
0 & u(y) \\
\kappa u^{*}(-y) e^{2 i z(x-y)} & 0
\end{array}\right)\left(\begin{array}{l}
f_1(y) \\
f_2(y)
\end{array}\right)dy,
$$
it can be seen from this that
$$
|\Gamma[f](x)| \leq \int_{-\infty}^{x}\left|u(y)\right| d y\|f\|_{\left(L^{\infty}(-\infty, 0]\right)},
$$
which means that $\Gamma$ is a bounded linear operator on $L^{\infty}(-\infty, 0]$. In addition, for $n\in \mathbb{N}$ through mathematical induction, it can be concluded that:
$$
\left|\Gamma^{n} f(x)\right| \leq \frac{1}{n !}\left(\int_{-\infty}^{x}\left|u(y)\right| d y\right)^{n}\|f\|_{\left(L^{\infty}(-\infty, 0]\right)}.
$$
Therefore, the characteristic function $\Phi_1^{-}(x,z)$ can be written in the following  series form
\begin{equation}
\Phi_1^{-}(x,z)=\sum_{n=0}^{\infty} \Gamma^{n}\left(\begin{array}{l}
1 \\
0
\end{array}\right),~~~~x\in(-\infty, 0]\label{p1}
\end{equation}
and it is uniformly convergent. Thus, it is deduced for $z\in \mathbb{R},x\in(-\infty, 0]$
$$
\left|\Phi^{-}_{1}(x, z)\right| \leq e^{\left\|u\right\|_{L^{1}(\mathbb{R})}}.
$$
Similarly, we can obtain
$$
\begin{aligned}
&\left|\Phi^{-}_{2}(x,z)\right| \leq e^{\left\|u\right\|_{L^{1}(\mathbb{R})}},\\
&\left|\Phi^{+}_{1}(x,z)\right| \leq e^{\left\|u\right\|_{L^{1}(\mathbb{R})}}, \\
&\left|\Phi^{+}_{2}(x,z)\right| \leq e^{\left\|u\right\|_{L^{1}(\mathbb{R})}}.
\end{aligned}
$$
In addition, according to the uniform convergence property of series (\ref{p1}), we have
$$
\begin{aligned}
\Phi^{-}_{1,z}(x, z) &=\sum_{n=0}^{\infty} \partial_{z}\Gamma^{n}\left(\begin{array}{l}
1 \\
0
\end{array}\right)=\sum_{n=1}^{\infty} \sum_{m=0}^{n} \Gamma^{m} \Gamma^{\prime} \Gamma^{n-m-1}\left(\begin{array}{l}
1 \\
0
\end{array}\right),
\end{aligned}
$$
where the form of $\Gamma^{\prime}$  is
$$
\Gamma^{\prime}[f](x)=\int_{-\infty}^{x}\left(\begin{array}{c}
u(y)f_2 \\
2 i\kappa (x-y) u^{*}(-y) e^{2 i z(x-y)}f_1
\end{array}\right) d y.
$$
It is also easy to get that the linear operator $\Gamma^{\prime}$  is bounded on $x\in(-\infty, 0]$
$$
\left\|\Gamma^{\prime}\right\|_{L^{\infty}(-\infty, 0]} \leq \left\|u\right\|_{L^{1,1}(\mathbb{R})},
$$
so there is
$$
\begin{aligned}
\Phi^{-}_{1,z}(x, z)& \leq \sum_{n=0}^{\infty}\left|\partial_{z}\Gamma^{n}\left(\begin{array}{l}
1 \\
0
\end{array}\right)\right|\leq \sum_{n=1}^{\infty} \frac{2\left\|u\right\|^{n-1}_{L^{1}(\mathbb{R})}}{(n-1) !} \left\|u\right\|_{L^{1,1}(\mathbb{R})} \\
& \leq 2 \left\|u\right\|_{L^{1,1}(\mathbb{R})} e^{\left\|u\right\|_{L^{1}(\mathbb{R})}}
\end{aligned}
$$
for  $z \in \mathbb{R}$  and $x\in(-\infty, 0]$.

In similar steps, we can also get
$$
\Phi^{-}_{2,z}(x, z)\leq  2 \left\|u\right\|_{L^{1,1}(\mathbb{R})} e^{\left\|u\right\|_{L^{1}(\mathbb{R})}},
$$
$$
\Phi^{+}_{1,z}(x, z)\leq 2 \left\|u\right\|_{L^{1,1}(\mathbb{R})} e^{\left\|u\right\|_{L^{1}(\mathbb{R})}},
$$
$$
\Phi^{+}_{2,z}(x, z)\leq 2 \left\|u\right\|_{L^{1,1}(\mathbb{R})} e^{\left\|u\right\|_{L^{1}(\mathbb{R})}}.
$$
On the basis of the expression of Eq.(\ref{rz}), $(\Phi_{ij}^{\pm})(x, z)(i,j=1,2)$ need to be estimated below. As can be seen from the expression of $\Phi_{1}^{-}(x,z)$ that
$$
\Phi^{-}_{21}(x, z)=\int_{-\infty}^{x}\kappa u^{*}(-y) e^{2 i z(x-y)}\Phi^{-}_{11}(y, z)dy,
$$
$\forall \sigma(z) \in C_{0}\left(\mathbb{R}\right)$, we compute
$$
\begin{aligned}
&\left|\int_{\mathbb{R}} \sigma(z) \int_{-\infty}^{x} \kappa u^{*}(-y)e^{2 i z(x-y)}\Phi_{11}^{-}(x, z) d y d z\right| \\
& \leq  \frac{1}{\sqrt{2}} e^{\left\|u\right\|_{L^{1}(\mathbb{R})}}\left\|u\right\|_{L^{2}(\mathbb{R})}\|\hat{\sigma}\|_{L^{2}(\mathbb{R})},
\end{aligned}
$$
where $\hat{\sigma}$ is the Fourier transform of $\sigma(z)$. This shows that
$$
\left\|\Phi_{21}^{-}(x, z)\right\|_{L_{z}^{2}(\mathbb{R})} \leq \frac{1}{\sqrt{2}}\left\|u\right\|_{L^{2}(\mathbb{R})} e^{\left\|u\right\|_{L^{1}(\mathbb{R})}}.
$$
Similarly, the same estimation can be obtained for $\Phi_{12}^{-}(x, z),\Phi_{21}^{+}(x, z)$ and $\Phi_{12}^{+}(x, z)$.
Further, the derivative expression of $\Phi_{21}^{-}(x, z)$ over $z$ is
$$
\Phi_{21,z}^{-}(x, z)=\int_{-\infty}^{x}\kappa u^{*}(-y) e^{2 i z(x-y)}\Phi^{-}_{11,z}(y, z)dy+\int_{-\infty}^{x}2 i (x-y)\kappa u^{*}(-y) e^{2 i z(x-y)}\Phi^{-}_{11}(y, z)dy,
$$
following the above steps, it is easy to get
$$
\begin{aligned}
\left\|\Phi^{-}_{21,z}(x, z)\right\| L_{z}^{2}(\mathbb{R}) \leq & \sqrt{2}\left\|u\right\|_{L^{1,1}(\mathbb{R})}\left\|q\right\|_{L^{2}(\mathbb{R})} e^{\left\|u\right\|_{L^{1}(\mathbb{R})}} +\sqrt{2}\left\|u\right\|_{L^{2, \frac{1}{2}(\mathbb{R})}} e^{\left\|u\right\|_{L^{1}(\mathbb{R})}}.
\end{aligned}
$$
For $\Phi^{-}_{12,z}(x, z),\Phi^{+}_{21,z}(x, z)$ and $\Phi^{+}_{12,z}(x, z)$ can be obtained similarly.

Eq.(\ref{rz}) implies that
$$
\begin{aligned}
&s_{11}(z)-1=(\Phi^{-}_{11}-1)(\Phi^{+}_{22}-1)+(\Phi^{-}_{11}-1)+(\Phi^{+}_{22}-1)-\Phi^{-}_{21}\Phi^{+}_{12},\\
&s_{11,z}(z)=\Phi^{-}_{11,z}\Phi^{+}_{22}+\Phi^{-}_{11}\Phi^{+}_{22,z}-\Phi^{-}_{21,z}\Phi^{+}_{12}-\Phi^{-}_{21}\Phi^{+}_{12,z}.\\
\end{aligned}
$$
It can be seen from integral Eq.(\ref{psz}) that the expression of $\Phi^{-}_{11}-1$ and $\Phi^{+}_{22}-1$ are
$$
\begin{aligned}
&\Phi^{-}_{11}-1=\int_{-\infty}^{x}u(y)\Phi_{21}^{-}dy,\\
&\Phi^{+}_{22}-1=\int_{\infty}^{x}\kappa u^{*}(y)\Phi_{12}^{+}dy.
\end{aligned}
$$
As we all know, the scattering data $s_{11}(z)$ is independent of $x$. Let $x = 0$ and then use the above estimation
\begin{equation}\label{sj1}
\left|s_{11}(z)-1\right| \leq\left\|u\right\|_{L^{1}(\mathbb{R})}^{2} e^{2\left\|u\right\|_{L^{1}(\mathbb{R})}}+2\left\|u\right\|_{L^{1}(\mathbb{R})} e^{\left\|u\right\|_{L^{1}(\mathbb{R})}}+e^{2\left\|u\right\|_{L^{1}(\mathbb{R})}},
\end{equation}
and
\begin{equation}\label{s11z}
\left|s_{11,z}(z)\right| \leq 8\left\|u\right\|_{L^{1,1}(\mathbb{R})} e^{2\left\|u\right\|_{L^{1}(\mathbb{R})}}.
\end{equation}
Eq.(\ref{rz}) also shows  that
$$
\begin{aligned}
&s_{21}(z)=e^{2izx}(\Phi^{-}_{12}\Phi^{+}_{22}-\Phi^{-}_{22}\Phi^{+}_{12}),\\
&s_{21,z}(z)=2ixe^{2izx}(\Phi^{-}_{12}\Phi^{+}_{22}-\Phi^{-}_{22}\Phi^{+}_{12})+e^{2izx}(\Phi^{-}_{12,z}\Phi^{+}_{22}
+\Phi^{-}_{12}\Phi^{+}_{22,z}-\Phi^{-}_{22,z}\Phi^{+}_{12}-\Phi^{-}_{22}\Phi^{+}_{12,z}).
\end{aligned}
$$
Similarly, $s_{12}(z)$ is independent of $x$, so it can also be processed at $x = 0$, using the above steps, the following estimates can be obtained
\begin{equation}\label{s21}
\begin{aligned}
&\|s_{21}\| _{L_{z}^{2}(\mathbb{R})}  \leq 2\left\|u\right\|_{L^{2}(\mathbb{R})} e^{2\left\|u\right\|_{L^{1}(\mathbb{R})}}, \\
&\left\|s_{21,z}\right\|_{L_{z}^{2}(\mathbb{R})} \leq 4 \left\|u\right\|_{L^{2}(\mathbb{R})}\left\|u\right\|_{L^{1,1}(\mathbb{R})} e^{2\left\|u\right\|_{L^{1}(\mathbb{R})}}
+2 \sqrt{2}e^{2\left\|u\right\|_{L^{1}(\mathbb{R})}}\left(\left\|u\right\|_{L^{1,1}(\mathbb{R})}\left\|u\right\|_{L^{2}(\mathbb{R})}+\left\|u\right\|_{L^{2, \frac{1}{2}}(\mathbb{R})}\right).
\end{aligned}
\end{equation}
In summary, based on Eqs. (\ref{sj1}), (\ref{s11z}) and (\ref{s21}),  we can get
$$
\begin{aligned}
&s_{11}(z) \in L^{2}(\mathbb{R}),~~~s_{21}(z) \in L^{2}(\mathbb{R}), \\
&s_{11,z}(z) \in L^{\infty}(\mathbb{R}), ~~s_{21,z}(z) \in L^{2}(\mathbb{R}).
\end{aligned}
$$
So there
$$
r(z) \in L^{2}(\mathbb{R}),~~~r_z(z) \in L^{2}(\mathbb{R}).
$$

For $\tilde{r}(z)$, according to symmetry (\ref{s12z}) and Eq.(\ref{rz}), we have
$$
\tilde{r}(z)=\frac{s_{12}(z)}{s_{22}(z)}=\frac{\kappa s^{*}_{21}(-z^{*})}{s_{22}(z)},~~~z\in \mathbb{R}
$$
and
$$
\tilde{r}_z(z)=\frac{\kappa s^{*}_{21,z}(-z^{*})s_{22}(z)-\kappa s^{*}_{21}(-z^{*})s_{22,z}(z)}{s^2_{22}(z)},~~~z\in \mathbb{R}
$$
Using Eq. (\ref{psz}), similar steps can be derived $s_{22}(z)\in L^{2}(\mathbb{R}),s_{22,z}(z)\in L^{2}(\mathbb{R})$, which lead to $\tilde{r}(z)\in L^{2}(\mathbb{R}), \tilde{r}_z(z)\in L^{2}(\mathbb{R})$.

For the initial value at any other time $t_0>0$, note that $\hat{\Psi}^{\pm}(x,t,z)$ is the Jost solution at the corresponding time
\begin{equation}\label{hp}
\hat{\Psi}^{\pm} \sim e^{ e^{-izx\sigma_3-(2i\alpha z^2+4i\beta z^3)(t-t_0)\sigma_3}}, \quad x \rightarrow \pm \infty.
\end{equation}
Referring to Eqs. (\ref{ps}), (\ref{s1122}) and (\ref{rz}), we can get
\begin{equation}\label{hr}
\hat{r}(z)=\frac{det(\hat{\Psi}^{+}_1,\hat{\Psi}^{-}_1)}{det(\hat{\Psi}^{-}_1,\hat{\Psi}^{+}_2)},
\end{equation}
due to the uniqueness of Jost solution, it can be obtained from Eqs. (\ref{pp}), (\ref{rz}), (\ref{hp}) and (\ref{hr}) that
$$
\hat{r}(z)=r(z) e^{(4\alpha z^2+8\beta z^3)t_{0}}.
$$
Through the above simple calculation, it can be seen that $r(z)$ continues to exist in $H^{1,1}(\mathbb{R})$, and the same conclusion is reached for $\tilde{r}(z)$.
\end{proof}

\section{The construction of a RHP}
Based on the above analysis of Lax pair, we define the following sectionally meromorphic matrices $M(x,t,z)$
\begin{equation}
M(x, t, z)=\left\{\begin{array}{ll}
\left(\frac{\Phi^{-}_{1}}{s_{11}(z)}, \Phi^{+}_{2}\right), & k \in \mathbb{C}_{+}, \\
\left(\Phi^{+}_{1}, \frac{\Phi^{-}_{2}}{s_{22}(z)}\right), & k \in \mathbb{C}_{-},
\end{array}\right.
\end{equation}
which meets the following RHP:

\begin{rhp}\label{r1}
Find a matrix-valued function $M(x,t,z)$ which satisfies:

$\bullet$  Analyticity: $M(x,t,z)$ is meromorphic in $\mathbb{C}\backslash\mathbb{R}$ ;

$\bullet$   Symmetry: $M(x,t,z)=\sigma_{0} M^{*}\left(-x,t,-z^{*}\right)\sigma_{0}^{-1};$

$\bullet$ Jump condition: $M(x,t,z)$ has continuous boundary values $M^{\pm}(x,t,z)$ on $\mathbb{R}$ and
\begin{equation}
M^{+}(x, t, z)=M^{-}(x, t, z) V(z), \quad z \in \mathbb{R},
\end{equation}
where
\begin{equation}
V(z)=\left(\begin{array}{cc}
1-\kappa r(z)\tilde{r}(z)& -\kappa\tilde{r}(z) e^{-2 i t \theta(z)} \\
r(z) e^{2 i t \theta(z)} & 1
\end{array}\right),~~~\theta=z\frac{x}{t}+2\alpha z^2+4\beta z^3;\label{vv}
\end{equation}

$\bullet$   Asymptotic behaviors:
$$
M(x, t, z)=I+\mathcal{O}\left(z^{-1}\right), \quad z \rightarrow \infty;
$$
\end{rhp}
\underline The solution of the nonlocal Hirota eqaution (\ref{nhe}) can be expressed by
\begin{equation}
u(x,t)= 2 i\lim _{z \rightarrow \infty} (z M)_{12}(x, t, z).
\end{equation}

\section{Conjugation}
Through the form of the jump matrix, it can be found that the long-term asymptotic of RHP (\ref{r1}) is affected by the growth and attenuation of the exponential function $e^{\pm2 i t \theta(k)}$. Therefore, it is necessary to deal with the oscillation term in the jump matrix in Eq. (\ref{vv}) and decompose the jump matrix according to the sign change diagram of $Re(i\theta)$  to ensure that any jump matrix is bounded in a given region. According to $\theta^{\prime}(z) = 0$ and $\theta^{\prime\prime}(z) \neq 0$, we can get two stationary phase points:
 $$
 z_1=\frac{-\alpha-\sqrt{\alpha^{2}-3 \beta \frac{x}{t}}}{6 \beta},~~~ z_2=\frac{-\alpha+\sqrt{\alpha^{2}-3 \beta \frac{x}{t}}}{6 \beta}
 $$
 where $\alpha^{2}-3 \beta \frac{x}{t}>0$,
 and the sign distribution of $ Re(i\theta)$ is obtained, which is shown in Fig. 2.
 \begin{figure}
\center
\begin{tikzpicture}\usetikzlibrary{arrows}
\coordinate [label=0: ] ()at (0,2);
\coordinate [label=0: ] ()at (0,4);
\coordinate [label=0: Re $z$] ()at (3,-0.2);
\coordinate [label=0:] ()at (2,0.1);
\coordinate [label=0:] ()at (-2.6,0.1);
\coordinate [label=0: $Re(i\theta)>0$] ()at (-0.82,1.2);
\path [fill=pink] (-4,3)--(-1.5,3) to
(-1.5,0)--(-4,0);
\path [fill=pink] (4,3)--(1.5,3) to
(1.5,0)--(4,0);
\path [fill=pink] (-1.5,0)--(1.5,0) to
(1.5,-3)--(-1.5,-3);
\coordinate [label=0: $Re(i\theta)<0$] ()at (-3.6,1.2);
\coordinate [label=0: $Re(i\theta)<0$] ()at (1.8,1.2);
\coordinate [label=0: $Re(i\theta)<0$] ()at (-0.82,-1);
\coordinate [label=0: $Re(i\theta)>0$] ()at (-3.6,-1);
\coordinate [label=0: $Re(i\theta)>0$] ()at (1.8,-1);
\coordinate [label=0: $\bullet$] ()at (-1.7,0);
\coordinate [label=0: $\bullet$] ()at (1.3,0);
\coordinate [label=0: $z_1$] ()at (-2,-0.2);
\coordinate [label=0: $z_2$] ()at (1.6,-0.2);
\draw[->, ] (3.5,0)--(4,0);
\draw[thin](4,0)--(-4,0);
\draw[thin](-1.5,3)--(-1.5,-3);
\draw[thin](1.5,3)--(1.5,-3);
\end{tikzpicture}\label{tt2}
\caption{\small Symbol distribution image of $Re(i\theta)$. }
\end{figure}
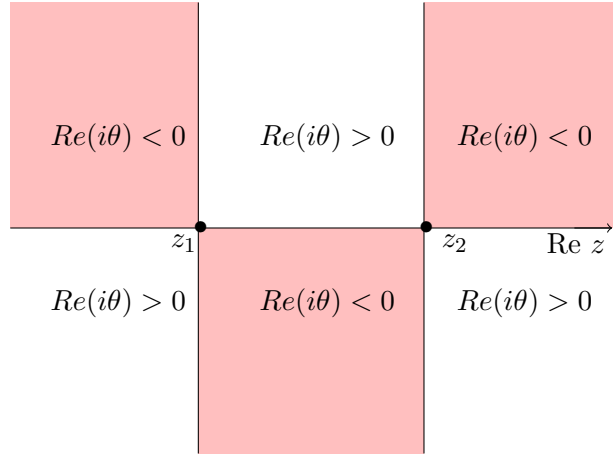

The jump matrix $V(z)$ in RHP(\ref{r1}) has the following decomposition
$$
V(z)=\left(\begin{array}{cc}
1 & -{\kappa\tilde{r}}(z) e^{-2 i t \theta} \\
0 & 1
\end{array}\right)\left(\begin{array}{cc}
1 & 0 \\
r(z) e^{2 i t \theta} & 1
\end{array}\right),~~~ z<z_1\cup z>z_2,
$$
and
$$
V(z)=\left(\begin{array}{cc}
1 & 0 \\
\frac{r(z)}{1-\kappa r(z)\tilde{r}(z)} e^{2 i t \theta} & 1
\end{array}\right)\left(\begin{array}{cc}
1-\kappa r(z)\tilde{r}(z) & 0 \\
0 & \frac{1}{1-\kappa r(z)\tilde{r}(z)}
\end{array}\right)\left(\begin{array}{cc}
1 & \frac{-\kappa \tilde{r}(z)}{1-\kappa r(z)\tilde{r}(z)} e^{-2 i t \theta} \\
0 & 1
\end{array}\right), z_1<z<z_2.
$$
The next goal is to eliminate the diagonal matrix of $z$ in the interval $(z_1,z_2)$ by introducing a RHP about $\delta(z)$
\begin{rhp}\label{brh}
Find a scalar function $\delta(z)$ which satisfies:

$\bullet$  Analyticity: $\delta(z)$ is analytic in  $\mathbb{C}\backslash(z_1,z_2)$;

$\bullet$  Jump condition: $
\delta_{+}(z)=\delta_{-}(z)\left(1-\kappa r(z)\tilde{r}(z)\right), \quad  z_1<z<z_2;$

$\bullet$  Asymptotic behaviors: $\delta(k)=1+\mathcal{O}\left(\frac{1}{z}\right)$ and $|\arg (z-z_j)| \leq c<\pi$
\begin{equation}\label{dz}
\delta(z)=1-\frac{i}{z}\int_{z_1}^{z_2} \frac{\log \left(1-\kappa r(s)\tilde{r}(s)\right)}{2\pi } d s+O\left(z^{-2}\right).~|z|\rightarrow\infty,
\end{equation}
\end{rhp}
Using the Plemelij formula, it is easy to write the unique solution of the above RHP \ref{brh} as
$$
\delta(z)=\exp \left[\frac{1}{2 \pi i} \int_{z_1}^{z_2} \frac{\log \left(1-\kappa r(s)\tilde{r}(s)\right)}{s-z} d s\right]=\exp \left[i \int_{z_1}^{z_2} \frac{\nu(s)}{s-z} d s\right].
$$
When $z \rightarrow z_1,z_2$ along any ray $L_{\phi}=\left\{z_{1}+ e^{i \phi}, z_{2}+e^{i \phi},-\pi<\phi<\pi\right\}$
 $$
\begin{aligned}
&\left|\delta(z)-\delta_{1}\left(z_1\right)\left(z-z_1\right)^{i \nu\left(z_1\right)}\right| \leq c\|r\|_{H^{1}(\mathbb{R})}\left|z-z_1\right|^{\frac{1}{2}-Im\nu(z_1) },\\
&\left|\delta(z)-\delta_{2}\left(z_2\right)\left(z-z_2\right)^{i \nu\left(z_2\right)}\right| \leq c\|r\|_{H^{1}(\mathbb{R})}\left|z-z_2\right|^{\frac{1}{2}-Im\nu(z_2) },
\end{aligned}
$$
where
$$
\begin{aligned}
&\delta_{1}\left(z_1\right)= e^{i \beta\left(z_1, z_1\right)},~~~\delta_{2}\left(z_2\right)= e^{i \beta\left(z_2, z_2\right)},\\
&\beta(z,z_1)=\left(i \int_{z_1}^{z_1+1} \frac{\nu\left(z_1\right)}{k-z} d k+i \int_{z_1}^{z_2} \frac{\nu(k)-\chi_1(k) \nu\left(z_1\right)}{k-z} d k\right),\\
&\beta(z,z_2)=\left(i \int_{z_2-1}^{z_2} \frac{\nu\left(z_2\right)}{k-z} d k+i \int_{z_1}^{z_2} \frac{\nu(k)-\chi_2(k) \nu\left(z_2\right)}{k-z} d k\right),\\
&\chi_{1}(z)=\left\{\begin{array}{lc}
1, & z_1<z<z_1+1, \\
0, &  { elsewhere },
\end{array} \quad \chi_{2}(z)=\left\{\begin{array}{lc}
1, & z_2-1<z<z_2 ,\\
0, & { elsewhere }.
\end{array}\right.\right.
\end{aligned}
$$
\begin{remark}
 Suppose that $|\arg (1-\kappa r(z)\tilde{r}(z))|<\pi$ and $|\operatorname{Im} \nu(z)|<\frac{1}{2}$ for $z\in \mathbb{R}$ in order to guarantee that $log(1-\kappa r(z)\tilde{r}(z))$ is single valued and that $\delta(z)$ is square integrable at the singularity $z_1$ and $z_2$.
\end{remark}
Then make the following changes in the interval $[z_1,z_2]$.
\begin{equation}
M^{(1)}(z)=M(z) \delta(z)^{-\sigma_{3}},
\end{equation}
Based on the above analysis, the RHP of $M^{(1)}(z)$ can be obtained as follows:
\begin{rhp}\label{4r}
Find a matrix $M^{(1)}(z)$ that satisfies:

$\bullet$ Analyticity: $M^{(1)}$ is analytic within $\mathbb{C}\backslash \mathbb{R};$

$\bullet$ Jump condition: $M^{(1)}_{+}( z)=M^{(1)}_{-}( z) V^{(1)}(z), \quad z \in \mathbb{R},$

where
$$
V^{(1)}(z)=\left\{\begin{array}{c}
\left(\begin{array}{cc}
1 & 0 \\
\frac{r(z)}{1-\kappa r(z)\tilde{r}(z)} \delta_{-}^{-2} e^{2 i t \theta(z)} & 1
\end{array}\right)\left(\begin{array}{cc}
1 & \frac{-\kappa \tilde{r}(z)}{1-\kappa r(z)\tilde{r}(z)} \delta_{+}^{2} e^{-2 i t \theta(z)} \\
0 & 1
\end{array}\right), \quad z_1<z<z_2, \\
\left(\begin{array}{cc}
1 & -\kappa \tilde{r}(z) \delta^{2} e^{-2 i t \theta(z)} \\
0 & 1
\end{array}\right)\left(\begin{array}{cc}
1 & 0 \\
r(z) \delta^{-2} e^{2 i t \theta(z)} & 1
\end{array}\right), \quad z>z_2\cup z<z_1;
\end{array}\right.
$$

$\bullet$  Normalization: $M^{(1)}(z)=I+\mathcal{O}\left(z^{-1}\right), \quad z \rightarrow \infty;$

\end{rhp}
In addition, due to $\delta(z)^{-\sigma3}\rightarrow I$, as $z\rightarrow\infty$, so the solution of the nonlocal Hirota equation (2.1) can be expressed as
\begin{equation}
u(x,t)= 2 i\lim _{z \rightarrow \infty} (z M^{(1)}(z))_{12}.
\end{equation}

\section{Contour deformation}
In this section, we mainly open the jump line at the  stationary phase points and deform $M^{(1)}(z)$ to get a $\bar{\partial}$-RHP. The new contour $\Sigma$ is changed to 8 routes composed of ${z_1+e^{i \phi} \mathbb{R}^{+}}\cup{z_{2}+e^{i \phi} \mathbb{R}^{+}},\phi=\{\frac{i\pi}{4}, \frac{3i\pi}{4},\frac{5i\pi}{4},\frac{7i\pi}{4}\}$, as shown in Fig. 3.
\begin{equation}
\Sigma=\Sigma_{1} \cup \Sigma_{2} \cup \Sigma_{3} \cup \Sigma_{4} \cup \Sigma_{5} \cup \Sigma_{6} \cup \Sigma_{7} \cup \Sigma_{8}.
\end{equation}
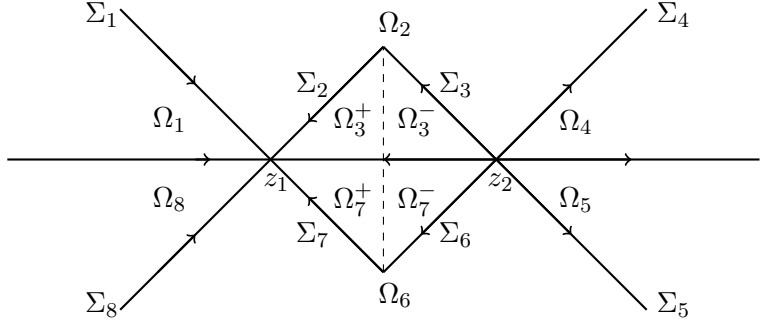
\begin{figure}\label{tt3}
\center
\begin{tikzpicture}\usetikzlibrary{arrows}
\coordinate [label=0:] ()at (2,0.1);
\coordinate [label=0:] ()at (-2.6,0.1);
\coordinate [label=0: $z_1$] ()at (-1.73,-0.3);
\coordinate [label=0: $z_2$] ()at (1.25,-0.3);
\coordinate [label=0: $\Sigma_1$] ()at (-4.1,1.93);
\coordinate [label=0: $\Sigma_2$] ()at (-1.3,1);
\coordinate [label=0: $\Sigma_3$] ()at (0.6,1);
\coordinate [label=0: $\Sigma_4$] ()at (3.5,1.93);
\coordinate [label=0: $\Sigma_8$] ()at (-4.1,-1.93);
\coordinate [label=0: $\Sigma_7$] ()at (-1.3,-1);
\coordinate [label=0: $\Sigma_6$] ()at (0.6,-1);
\coordinate [label=0: $\Sigma_5$] ()at (3.5,-1.93);
\coordinate [label=0: $\Omega_1$] ()at (-3.2,0.53);
\coordinate [label=0: $\Omega_2$] ()at (-0.2,1.8);
\coordinate [label=0: $\Omega_3^{+}$] ()at (-0.8,0.53);
\coordinate [label=0: $\Omega_3^{-}$] ()at (0.05,0.53);
\coordinate [label=0: $\Omega_4$] ()at (2.2,0.5);
\coordinate [label=0: $\Omega_5$] ()at (2.2,-0.53);
\coordinate [label=0: $\Omega_6$] ()at (-0.2,-1.8);
\coordinate [label=0: $\Omega_7^{+}$] ()at (-0.8,-0.53);
\coordinate [label=0: $\Omega_7^{-}$] ()at (0.05,-0.53);
\coordinate [label=0: $\Omega_8$] ()at (-3.2,-0.53);
\draw[thick](0,1.5)--(-3.5,-2);
\draw[thick](0,-1.5)--(-3.5,2);
\draw[thick](0,1.5)--(3.5,-2);
\draw[thick](0,-1.5)--(3.5,2);
\draw[thick](5,0)--(-5,0);
\draw[->,thick] (2.5,0)--(3.3,0);
\draw[->,thick] (2.5,0)--(0,0);
\draw[->,thick] (-2.5,0)--(-2.3,0);
\draw[->,thick] (1.5,0)--(2.5,1);
\draw[->,thick] (1.5,0)--(2.5,-1);
\draw[->,thick] (1.5,0)--(0.5,1);
\draw[->,thick] (1.5,0)--(0.5,-1);
\draw[->,thick] (0,-1.5)--(-1,-0.5);
\draw[->,thick] (0,1.5)--(-1,0.5);
\draw[->,thick] (-3,1.5)--(-2.5,1);
\draw[->,thick] (-3,-1.5)--(-2.5,-1);
\draw[dashed](0,-1.5)--(0,1.5);
\end{tikzpicture}
\caption{\small  The jump contour $\Sigma$ and domains $\Omega_j$. }
\end{figure}

There is a scalar function $R_j \rightarrow \mathbb{C}$ that satisfies the following boundary conditions:
\begin{equation}\label{R1}
R_{1}(z)= \begin{cases}r(z) \delta^{-2}(z), & z \in\left(-\infty,z_1 \right), \\ r\left(z_1\right) \delta_{1}^{-2}(z_1)\left(z-z_1\right)^{-2 i v\left(z_1\right)}, & z \in \Sigma_{1},\end{cases}
\end{equation}
\begin{equation}\label{R2}
R_{3^{+}}(z)= \begin{cases}\frac{-\kappa\tilde{r}(z) \delta_{+}^{2}(z)}{1-\kappa r(z)\tilde{r}(z)}, & z \in\left(z_1, z_2\right), \\ \frac{-\kappa \tilde{r}(z_1) \delta_{1}^{2}(z_1)}{1-\kappa r(z_1)\tilde{r}(z_1)}\left(z-z_1\right)^{2 i v\left(z_1\right)}, & z \in \Sigma_{2},\end{cases}
\end{equation}
\begin{equation}\label{R3}
R_{3^{-}}(z)= \begin{cases}\frac{-\kappa \tilde{r}(z) \delta_{+}^{2}(z)}{1-\kappa r(z)\tilde{r}(z)}, & z \in\left(z_1, z_2\right), \\ \frac{-\kappa \tilde{r}(z_2) \delta_{2}^{2}(z_2)}{1-\kappa r(z_2)\tilde{r}(z_2)}\left(z-z_2\right)^{2 i v\left(z_2\right)}, & z \in \Sigma_{3},\end{cases}
\end{equation}
\begin{equation}\label{R4}
R_{4}(z)= \begin{cases}r(z) \delta^{-2}(z), & z \in\left(z_2,\infty \right), \\ r\left(z_2\right) \delta_{2}^{-2}(z_2)\left(z-z_2\right)^{-2 i v\left(z_2\right)}, & z \in \Sigma_{4},\end{cases}
\end{equation}
\begin{equation}\label{R5}
R_{5}(z)= \begin{cases}-\kappa \tilde{r}(z) \delta^{2}(z), & z \in\left(z_2,\infty \right),\\ -\kappa \tilde{r}\left(z_2\right) \delta_2^{2}(z_2)\left(z-z_2\right)^{2 i v\left(z_2\right)}, & z \in \Sigma_{5},\end{cases}
\end{equation}
\begin{equation}\label{R6}
R_{7^{-}}(z)= \begin{cases}\frac{r(z) \delta_{-}^{-2}(z)}{1-\kappa r(z)\tilde{r}(z)}, & z \in\left(z_1, z_2\right), \\ \frac{r\left(z_2\right) \delta_2^{2}(z_2)}{1-\kappa r(z_2)\tilde{r}(z_2)}\left(z-z_2\right)^{-2 i v\left(z_2\right)}, & z \in \Sigma_{6},\end{cases}
\end{equation}
\begin{equation}\label{R7}
R_{7^{+}}(z)= \begin{cases}\frac{r(z) \delta_{-}^{-2}(z)}{1-\kappa r(z)\tilde{r}(z)}, & z \in\left(z_1, z_2\right), \\ \frac{r\left(z_1\right) \delta_1^{2}(z_1)}{1-\kappa r(z_1)\tilde{r}(z_1)}\left(z-z_1\right)^{-2 i v\left(z_1\right)}, & z \in \Sigma_{7},\end{cases}
\end{equation}
\begin{equation}\label{R8}
R_{8}(z)= \begin{cases}-\kappa \tilde{r}(z) \delta^{2}(z), & z \in\left(-\infty,z_1\right),\\ -\kappa\tilde{r}\left(z_1\right) \delta_1^{2}(z_1)\left(z-z_1\right)^{2 i v\left(z_1\right)}, & z \in \Sigma_{8},\end{cases}
\end{equation}
And meet the following estimates
\begin{equation}\label{rj}
\begin{aligned}
&\left|R_{j}\right| \lesssim c \sin ^{2}\left(\arg \left(z-z_1\right)\right)+\langle\operatorname{Re}(z)\rangle^{-\frac{1}{2}},~~j=1,3^{+},7^{+},8,\\
&\left|R_{j}\right| \lesssim c \sin ^{2}\left(\arg \left(z-z_2\right)\right)+\langle\operatorname{Re}(z)\rangle^{-\frac{1}{2}},~~j=3^{-},4,5,7^{-},\\
&\left|\bar{\partial} R_{j}(z)\right| \leq c\left(\left|p_j^{\prime}(\operatorname{Re} z)\right|+|z-z_1|^{-\frac{1}{2}-Im\nu(z_1)}\right),~~j=1,3^{+},7^{+},8,\\
&\left|\bar{\partial} R_{j}(z)\right| \leq c\left(\left|p_j^{\prime}(\operatorname{Re} z)\right|+|z-z_2|^{-\frac{1}{2}-Im\nu(z_2)}\right),~~j=3^{-},4,5,7^{-},\\
&p_1=p_4=r,~~p_{3^{+}}=p_{3^{-}}=\frac{-\kappa\tilde{r}}{1-\kappa r\tilde{r}},~~p_5=p_8=-\kappa \tilde{r},~~p_{7^{+}}=p_{7^{-}}=\frac{{r}}{1-\kappa r\tilde{r}}.
\end{aligned}
\end{equation}
On this basis, we can construct a matrix $\mathcal{R}^{(1)}(z)$:
$$
\mathcal{R}^{(1)}(z)= \begin{cases}\left(\begin{array}{cc}
1 & 0 \\
(-1)^{m_j}R_{j}(z) e^{2 i t \theta} & 1
\end{array}\right), & z \in \Omega_{ j},\quad j=1,4,7^{+},7^{-} \\
\left(\begin{array}{cc}
1 & (-1)^{m_j}R_{j}(z) e^{-2 i t \theta} \\
0 & 1
\end{array}\right), & z \in \Omega_{j},\quad j=3^{-},3^{+},5,8 \\
\left(\begin{array}{ll}
1 & 0 \\
0 & 1
\end{array}\right), & z \in \Omega_{2} \cup \Omega_{6},\end{cases}
$$
where $m_{3^{+}}=m_4=m_{7^{-}}=m_8=1, m_1=m_5=m_{3^{-}}=m_{7^{+}}=0$.

Then we do the following transformation to get $M^{(2)}(z)$
\begin{equation}
M^{(2)}(z)=M^{(1)}(z) \mathcal{R}^{(1)}(z),
\end{equation}
where
 Matrix $M^{(2)}(z)$ satisfies the following mixed $\bar{\partial}$-$RHP$:
\begin{rhp}\label{5r}
Find a matrix $M^{(2)}(z)$  that satisfies the following properties:

$\bullet$ Analyticity: $M^{(2)}(z)$ is analytic within $\mathbb{C}\backslash \Sigma$;

$\bullet$ Normalization: $M^{(2)}(z)=I+\mathcal{O}\left(z^{-1}\right), \quad z \rightarrow \infty;$

$\bullet$ Jump condition: $M^{(2)}_{+}(z)=M^{(2)}_{-}(z) V^{(2)}(z), \quad z \in \Sigma,$
where
$$
V^{(2)}(z)= \begin{cases}\left(\begin{array}{cc}
1 & 0 \\
R_{1}(z) e^{2 i t \theta} & 1
\end{array}\right), & z \in \Sigma_{1}, \\
\left(\begin{array}{cc}
1 & -R^{+}_{3}(z) e^{-2 i t \theta} \\
0 & 1
\end{array}\right), & z \in \Sigma_{2}, \\
\left(\begin{array}{cc}
1 & R^{-}_{3}(z) e^{-2 i t \theta} \\
0 & 1
\end{array}\right), & z \in \Sigma_{3} ,\\
\left(\begin{array}{cc}
1 &  \\
-R_{4}(z) e^{2 i t \theta} & 1
\end{array}\right), & z \in \Sigma_{4}, \\
\left(\begin{array}{cc}
1 & R_{5}(z) e^{-2 i t \theta}\\
0  & 1
\end{array}\right), & z \in \Sigma_{5},\\
\left(\begin{array}{cc}
1 & 0\\
-R^{-}_{7}(z) e^{2 i t \theta}  & 1
\end{array}\right), & z \in \Sigma_{6},\\
\left(\begin{array}{cc}
1 & 0\\
R^{+}_{7}(z) e^{2 i t \theta}  & 1
\end{array}\right), & z \in \Sigma_{7},\\
\left(\begin{array}{cc}
1 & -R_{8}(z) e^{-2 i t \theta} \\
0 & 1
\end{array}\right), & z \in \Sigma_{8}.
 \end{cases}
$$

$\bullet$  $\bar{\partial}$ derivative:
\begin{equation}\label{4.1}
\bar{\partial} M^{(2)}(z)=M^{(2)}(z) \bar{\partial} \mathcal{R}^{(1)}(z), ~~z\in \mathbb{C}\backslash \Sigma
\end{equation}
where
$$
 \bar{\partial} \mathcal{R}^{(1)}(z)= \begin{cases}\left(\begin{array}{cc}
1 & 0 \\
(-1)^{m_j}\bar{\partial}R_{j}(z) e^{2 i t \theta} & 1
\end{array}\right), & z \in \Omega_{ j},\quad j=1,4,7^{+},7^{-} \\
\left(\begin{array}{cc}
1 & (-1)^{m_j}\bar{\partial}R_{j}(z) e^{-2 i t \theta} \\
0 & 1
\end{array}\right), & z \in \Omega_{j},\quad j=3^{-},3^{+},5,8 \\
\left(\begin{array}{ll}
0 & 0 \\
0 & 0
\end{array}\right), & z \in \Omega_{2} \cup \Omega_{6},\end{cases}
$$
where $m_{3^{+}}=m_4=m_{7^{-}}=m_8=1, m_1=m_5=m_{3^{-}}=m_{7^{+}}=0$.
\end{rhp}
The relationship between the solution of  the nonlocal Hirota equation  and $M^{(2)}(z)$ is
\begin{equation}
u(x,t)= 2 i\lim _{z \rightarrow \infty} (z M^{(2)}(z))_{12}.
\end{equation}

\section{The decomposition of the mixed $\bar{\partial}$-$RHP$ }
This section is to decompose the mixed $\bar{\partial}$-$RHP$ \ref{5r}. The main idea is based on whether $\bar{\partial} \mathcal{R}^{(1)}(z)$ is 0. Therefore, we can decompose $M^{(2)}(z)=E(z)M^{(2)}_{rhp}(z)$ as follows. Here, when $\bar{\partial} \mathcal{R}^{(1)}(z)=0$, it corresponds to a model RHP about $M^{(2)}_{rhp}(z)$, when $\bar{\partial} \mathcal{R}^{(1)}(z)\neq0$, it corresponds to a pure $\bar{\partial}$ problem about $E(z)$, and there is $ E(z)=M^{(2)}(z){M^{(2)}_{rhp}}^{-1}(z)$.

\subsection{The model RH problem}
This subsection focuses on the analysis of the model RHP about $M^{(2)}_{rhp}(z)$. There are two stationary phase points in the nonlocal Hirota equation, so two models RHP are needed to solve them respectively.

First, when $z\rightarrow z_1$, we can do the scale transformation into:
$$
\tilde{M}_{z_{1}}(k)=M_{rhp}^{(2)}(z), ~z=\frac{k}{\sqrt{8\left(6 \beta z_{1}+\alpha\right) t}}+z_{1},
$$
and $\tilde{M}_{z_{1}}(k)$ meets the following RHP:
\begin{rhp}\label{r6}
Find a matrix-valued function $\tilde{M}_{z_{1}}(k)$ with following properties:

$\bullet$ Analyticity: $\tilde{M}_{z_{1}}(k)$ is analytic within $\mathbb{C}\backslash \tilde{\Sigma},~~\tilde{\Sigma}=\Sigma-z_1;$

$\bullet$  Asymptotic behavior:
$\tilde{M}_{z_{1}}(k)=I+\mathcal{O}\left(k^{-1}\right), \quad k \rightarrow \infty,$

$\bullet$  Jump condition: $\tilde{M}^{+}_{z_{1}}(k)=\tilde{M}^{-}_{z_{1}}(k)\tilde{V}_{z_1}(k), \quad k \in \tilde{\Sigma};$

where
$$
\tilde{V}_{z_1}(k)= k^{i \nu(z_1) \hat{\sigma}_{3}} e^{-\frac{i k^{2}}{4} \hat{\sigma}_{3}}
\begin{cases}\left(\begin{array}{cc}
1 & 0 \\
-r(z_1)\varpi_{1}(z_1) & 1
\end{array}\right),  \\
\left(\begin{array}{cc}
1 & \frac{-\kappa \tilde{r}(z_1)}{1-\kappa r(z_1)\tilde{r}(z_1)}\varpi_1^{-1}(z_1) \\
0 & 1
\end{array}\right),  \\
\left(\begin{array}{cc}
1 & 0 \\
\frac{r(z_1)}{1-\kappa r(z_1)\tilde{r}(z_1)}\varpi_{1}(z) & 1
\end{array}\right),  \\
\left(\begin{array}{cc}
1 & \kappa \tilde{r}(z_1)\varpi_1^{-1}(z_1) \\
0 & 1
\end{array}\right).
\end{cases}
$$
and
$\varpi_1(z_1)=\delta^{-2}_{1}(z_1)(48t\beta z_1+8\alpha t)^{i \nu(z_1)} e^{8it\beta(z-z_1)^3-16it\beta z_1^3-4it\alpha z_1^2}.$
\end{rhp}
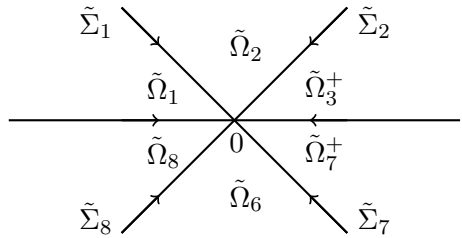
\begin{figure}[h]
\center
\begin{tikzpicture}\usetikzlibrary{arrows}
\coordinate [label=0:] ()at (2,0.1);
\coordinate [label=0:] ()at (-2.6,0.1);
\coordinate [label=0: $0$] ()at (-0.2,-0.3);
\coordinate [label=0: $\tilde{\Sigma}_1$] ()at (-2.2,1.3);
\coordinate [label=0: $\tilde{\Sigma}_2$] ()at (1.5,1.3);
\coordinate [label=0: $\tilde{\Sigma}_8$] ()at (-2.2,-1.3);
\coordinate [label=0: $\tilde{\Sigma}_7$] ()at (1.5,-1.3);
\coordinate [label=0: $\tilde{\Omega}_1$] ()at (-1.3,0.43);
\coordinate [label=0: $\tilde{\Omega}_2$] ()at (-0.2,1);
\coordinate [label=0: $\tilde{\Omega}_{3}^{+}$] ()at (0.8,0.43);
\coordinate [label=0: $\tilde{\Omega}_{7}^{+}$] ()at (0.8,-0.43);
\coordinate [label=0: $\tilde{\Omega}_8$] ()at (-1.3,-0.43);
\coordinate [label=0: $\tilde{\Omega}_6$] ()at (-0.2,-1);
\draw[thick](1.5,-1.5)--(-1.5,1.5);
\draw[thick](1.5,1.5)--(-1.5,-1.5);
\draw[thick](3,0)--(-3,0);
\draw[->,thick] (1.5,1.5)--(1,1);
\draw[->,thick] (1.5,-1.5)--(1,-1);
\draw[->,thick] (-1.5,1.5)--(-1,1);
\draw[->,thick] (-1.5,-1.5)--(-1,-1);
\draw[->,thick] (1.5,0)--(1,0);
\draw[->,thick] (-1.5,0)--(-1,0);
\end{tikzpicture}
\caption{\small  The jump contour $\tilde{\Sigma}$ and domains $\tilde{\Omega}_j$. }
\end{figure}

If we make the following change:
$$
\tilde{M}_{z_{1}}(k)=\mathcal{M}\mathcal{P}k^{-i \nu(z_1) {\sigma}_{3}} e^{\frac{i k^{2}}{4} {\sigma}_{3}},
$$
where $\mathcal{P}$ is a constant matrix ,
$$
\mathcal{P}=\begin{cases}\left(\begin{array}{cc}
1 & 0 \\
r(z_1)\varpi_1(z_1) & 1
\end{array}\right),~~~k\in \tilde{\Omega}_1  \\
\left(\begin{array}{cc}
1 & \frac{\kappa \tilde{r}(z_1)}{1-\kappa r(z_1)\tilde{r}(z_1)}\varpi_1^{-1}(z_1) \\
0 & 1
\end{array}\right),~~~k\in \tilde{\Omega}_3  \\
\left(\begin{array}{cc}
1 & 0 \\
-\frac{r(z_1)}{1-\kappa r(z_1)\tilde{r}(z_1)}\varpi_1(z_1) & 1
\end{array}\right), ~~~k\in \tilde{\Omega}_7 \\
\left(\begin{array}{cc}
1 & -\kappa \tilde{r}(z_1)\varpi_1^{-1}(z_1) \\
0 & 1
\end{array}\right), ~~~k\in \tilde{\Omega}_8\\
\left(\begin{array}{ll}
1 & 0 \\
0 & 1
\end{array}\right),~~~k\in \tilde{\Omega}_2\cup \tilde{\Omega}_6.
\end{cases}
$$
and $\mathcal{M}(k)$ satisfies the following RHP:
\begin{rhp}
Find a matrix value function $\mathcal{M}(k)$ admitting:

$\bullet$ Analyticity: $\mathcal{M}(k)$ is  holomorphic within $\mathbb{C}\backslash {\mathbb{R}};$

$\bullet$  Asymptotic behavior: $\mathcal{M}(k)k^{-i \nu(z_1) {\sigma}_{3}} e^{\frac{i k^{2}}{4} {\sigma}_{3}}= I+\mathcal{O}\left(k^{-1}\right), \quad k \rightarrow \infty,$

$\bullet$  Jump condition: $\mathcal{M}^{+}(k)=\mathcal{M}^{-}(k)\mathcal{V}(k), \quad k \in \mathbb{R};$
where
$$
\mathcal{V}(k)= \begin{cases}\left(\begin{array}{cc}
1-\kappa r(z_1)\tilde{r}(z_1) & -\kappa \tilde{r}(z_1)\varpi_1^{-1}(z_1) \\
r(z_1)\varpi_1(z_1) & 1
\end{array}\right).  \\
\end{cases}
$$
\end{rhp}
In fact, the above RHP can be transformed into the well-known Weber equation, and its solution is given by the parabolic cylindrical function:

\noindent
when $k\in \mathbb{C}_{+}$
$$
\begin{aligned}
&\mathcal{M}^{+}_{11}(k)=e^{-\frac{3 \pi \nu(z_1)}{4}} D_{a_1}\left(e^{-\frac{3 i \pi}{4}} k\right), \\
&\mathcal{M}^{+}_{12}=e^{\frac{3 \pi \nu(z_1)}{4}} \beta_{21}^{-1}(z_1)\left[\partial_{k} D_{-a_1}\left(e^{-\frac{\pi i}{4}} k\right)-\frac{i k}{2} D_{-a_1}\left(e^{-\frac{\pi i}{4}} k\right)\right], \\
&\mathcal{M}^{+}_{21}=e^{\frac{-3 \pi \nu(z_1)}{4}} \beta_{12}^{-1}(z_1)\left[\partial_{k} D_{a_1}\left(e^{-\frac{3 \pi i}{4}} k\right)+\frac{i k}{2} D_{a_1}\left(e^{-\frac{3 \pi i}{4}} k\right)\right], \\
&\mathcal{M}^{+}_{22}=e^{\frac{\pi \nu(z_1)}{4} } D_{-a_1}\left(e^{\frac{-i\pi}{4}} k\right).
\end{aligned}
$$
when $k\in \mathbb{C}_{-}$
$$
\begin{aligned}
&\mathcal{M}^{-}_{11}=e^{\frac{\pi \nu(z_1)}{4}} D_{a_1}\left(e^{\frac{\pi i }{4}} k\right), \\
&\mathcal{M}^{-}_{12}=e^{\frac{-3 \pi \nu(z_1)}{4}} \beta_{21}^{-1}(z_1)\left[\partial_{k} D_{-a_1}\left(e^{\frac{3 \pi i }{4}} k\right)-\frac{i k}{2} D_{-a_1}\left(e^{\frac{3 \pi i }{4}} k\right)\right], \\
&\mathcal{M}^{-}_{21}=e^{\frac{\pi \nu(z_1)}{4}} \beta_{12}^{-1}(z_1)\left[\partial_{k} D_{a_1}\left(e^{\frac{\pi i}{4} } k\right)+\frac{i k}{2} D_{a_1}\left(e^{\frac{\pi i}{4}} k\right)\right], \\
&\mathcal{M}^{-}_{22}=e^{\frac{-3 \pi \nu(z_1)}{4}} D_{a_1}\left(e^{\frac{3 \pi i}{4} } k\right),
\end{aligned}
$$
where
$$
\begin{aligned}
&\beta_{12}(z_1)=t^{Im\nu(z_1)}\frac{\sqrt{2 \pi} e^{\frac{\pi i}{4}} e^{-\frac{\pi \nu(z_1)}{2}}}{\vartheta(z_1)\Gamma(-i \nu(z_1))},\\
&\beta_{21}(z_1)=-t^{-Im\nu(z_1)}\frac{\sqrt{2 \pi} e^{-\frac{\pi i}{4}} e^{-\frac{\pi \nu(z_1)}{2}}}{\tilde{\vartheta}(z_1) \Gamma(i \nu(z_1))}, ~~~\beta_{12}(z_1)\beta_{21}(z_1)=\nu(z_1),\\
&\vartheta(z_1)=r(z_1)\delta^{-2}_{1}(z_1)e^{iRe\nu(z_1)In(48t\beta z_1+8\alpha t)+8it\beta(z-z_1)^3-16it\beta z_1^3-4it\alpha z_1^2},\\
&\tilde{\vartheta}(z_1)=\kappa\tilde{r}(z_1)\delta^{2}_{1}(z_1)e^{-iRe\nu(z_1)In(48t\beta z_1+8\alpha t)-8it\beta(z-z_1)^3+16it\beta z_1^3-4it\alpha z_1^2},
\end{aligned}
$$
and $D_{a_1}(k)$ is the  solution of
$$
\partial^{2}_{k}D_{a_1}(k)+(\frac{1}{2}-\frac{k^2}{4}+a_1)D_{a_1}(k)=0,~~~a_1=i\nu(z_1).
$$
The following scale transformation is still be considered, when $z\rightarrow z_2$,
$$
\hat{M}_{z_{2}}(k)=M_{rhp}^{(2)}(z),~~~ z=\frac{k}{\sqrt{8\left(6 \beta z_{2}+\alpha\right) t}}+z_{2},
$$
and $\hat{M}_{z_{2}}(k)$ meets the following RHP:
\begin{rhp}
Find a matrix-valued function $\hat{M}_{z_{2}}(k)$ with following properties:

$\bullet$ Analyticity: $\hat{M}_{z_{2}}(k)$ is analytic within $\mathbb{C}\backslash \hat{\Sigma},~~\hat{\Sigma}=\Sigma-z_2;$

$\bullet$  Asymptotic behavior:
$\hat{M}_{z_{2}}(k)=I+\mathcal{O}\left(k^{-1}\right), \quad k \rightarrow \infty,$

$\bullet$  Jump condition: $\hat{M}^{+}_{z_{1}}(k)=\tilde{M}^{-}_{z_{2}}(k)\hat{V}_{z_2}(k), \quad k \in \hat{\Sigma};$

where
$$
\hat{V}_{z_2}(k)= k^{i \nu(z_2) \hat{\sigma}_{3}} e^{-\frac{i k^{2}}{4} \hat{\sigma}_{3}}
\begin{cases}\left(\begin{array}{cc}
1 & 0 \\
r(z_2)\varpi_{2}(z_2) & 1
\end{array}\right),  \\
\left(\begin{array}{cc}
1 & \frac{\kappa\tilde{r}(z_2)}{1-\kappa r(z_2)\tilde{r}(z_1)}\varpi_2^{-1}(z_2) \\
0 & 1
\end{array}\right),  \\
\left(\begin{array}{cc}
1 & 0 \\
-\frac{r(z_2)}{1-\kappa r(z_2)\tilde{r}(z_2)}\varpi_{2}(z) & 1
\end{array}\right),  \\
\left(\begin{array}{cc}
1 & -\kappa\tilde{r}(z_2)\varpi_2^{-1}(z_2) \\
0 & 1
\end{array}\right),
\end{cases}
$$
and
$\varpi_2(z_2)=\delta^{-2}_{2}(z_2)(48t\beta z_2+8\alpha t)^{i \nu(z_2)} e^{8it\beta(z-z_2)^3-16it\beta z_2^3-4it\alpha z_2^2}.$
\end{rhp}
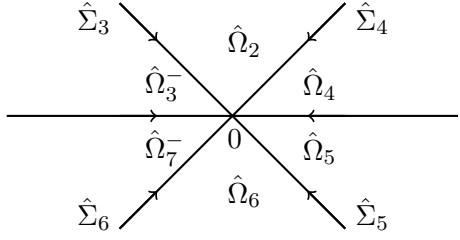
\begin{figure}
\center
\begin{tikzpicture}\usetikzlibrary{arrows}
\coordinate [label=0:] ()at (2,0.1);
\coordinate [label=0:] ()at (-2.6,0.1);
\coordinate [label=0: $0$] ()at (-0.2,-0.3);
\coordinate [label=0: $\hat{\Sigma}_3$] ()at (-2.2,1.3);
\coordinate [label=0: $\hat{\Sigma}_4$] ()at (1.5,1.3);
\coordinate [label=0: $\hat{\Sigma}_6$] ()at (-2.2,-1.3);
\coordinate [label=0: $\hat{\Sigma}_5$] ()at (1.5,-1.3);
\coordinate [label=0: $\hat{\Omega}_{3}^{-}$] ()at (-1.3,0.43);
\coordinate [label=0: $\hat{\Omega}_2$] ()at (-0.2,1);
\coordinate [label=0: $\hat{\Omega}_{4}$] ()at (0.8,0.43);
\coordinate [label=0: $\hat{\Omega}_{5}$] ()at (0.8,-0.43);
\coordinate [label=0: $\hat{\Omega}_{7}^{-}$] ()at (-1.3,-0.43);
\coordinate [label=0: $\hat{\Omega}_6$] ()at (-0.2,-1);
\draw[thick](1.5,-1.5)--(-1.5,1.5);
\draw[thick](1.5,1.5)--(-1.5,-1.5);
\draw[thick](3,0)--(-3,0);
\draw[->,thick] (1.5,1.5)--(1,1);
\draw[->,thick] (1.5,-1.5)--(1,-1);
\draw[->,thick] (-1.5,1.5)--(-1,1);
\draw[->,thick] (-1.5,-1.5)--(-1,-1);
\draw[->,thick] (1.5,0)--(1,0);
\draw[->,thick] (-1.5,0)--(-1,0);
\end{tikzpicture}
\caption{\small  The jump contour $\hat{\Sigma}$ and domains $\hat{\Omega}_j$. }
\end{figure}
If we make the following change:
$$
\hat{M}_{z_{2}}(k)=\mathcal{\hat{M}}\mathcal{\hat{P}}k^{-i \nu(z_2) {\sigma}_{3}} e^{\frac{i k^{2}}{4} {\sigma}_{3}},
$$
where $\mathcal{\hat{P}}$ is a constant matrix ,
$$
\mathcal{\hat{P}}=\begin{cases}\left(\begin{array}{cc}
1 & 0 \\
-r(z_2)\varpi_2(z_2) & 1
\end{array}\right),~~~k\in \hat{\Omega}_4  \\
\left(\begin{array}{cc}
1 & -\frac{\kappa \tilde{r}(z_2)}{1-\kappa r(z_2)\tilde{r}(z_2)}\varpi_2^{-1}(z_2) \\
0 & 1
\end{array}\right),~~~k\in \hat{\Omega}_3^{-}  \\
\left(\begin{array}{cc}
1 & 0 \\
\frac{r(z_2)}{1-\kappa r(z_2)\tilde{r}(z_2)}\varpi_2(z_2) & 1
\end{array}\right), ~~~k\in \hat{\Omega}_7^{-} \\
\left(\begin{array}{cc}
1 & \kappa \tilde{r}(z_2)\varpi_2^{-1}(z_2) \\
0 & 1
\end{array}\right), ~~~k\in\hat{\Omega}_5 \\
\left(\begin{array}{ll}
1 & 0 \\
0 & 1
\end{array}\right),~~~k\in \hat{\Omega}_2\cup \hat{\Omega}_6.
\end{cases}
$$
and $\mathcal{\hat{M}}(k)$ satisfies the following RHP:
\begin{rhp}
Find a matrix value function $\mathcal{\hat{M}}(k)$ admitting:

$\bullet$ Analyticity: $\mathcal{\hat{M}}(k)$ is  holomorphic within $\mathbb{C}\backslash {\mathbb{R}};$

$\bullet$  Asymptotic behavior: $\mathcal{\hat{M}}(k)k^{-i \nu(z_2) {\sigma}_{3}} e^{\frac{i k^{2}}{4} {\sigma}_{3}}= I+\mathcal{O}\left(k^{-1}\right), \quad k \rightarrow \infty,$

$\bullet$  Jump condition: $\mathcal{\hat{M}}^{+}(k)=\mathcal{\hat{M}}^{-}(k)\mathcal{\hat{V}}(k), \quad k \in \mathbb{R};$
where
$$
\mathcal{\hat{V}}(k)= \begin{cases}\left(\begin{array}{cc}
1-\kappa r(z_2)\tilde{r}(z_2) & -\kappa \tilde{r}(z_2)\varpi_2^{-1}(z_2) \\
r(z_2)\varpi_2(z_2) & 1
\end{array}\right).  \\
\end{cases}
$$
\end{rhp}
Similarly, we can write the above solution. When $k\in \mathbb{C}_{+}$
$$
\begin{aligned}
&\mathcal{\hat{M}}^{+}_{11}(k)=e^{-\frac{3 \pi \nu(z_2)}{4}} D_{a_2}\left(e^{-\frac{3 i \pi}{4}} k\right), \\
&\mathcal{\hat{M}}^{+}_{12}=e^{\frac{3 \pi \nu(z_2)}{4}} \beta_{21}^{-1}(z_2)\left[\partial_{k} D_{-a_2}\left(e^{-\frac{\pi i}{4}} k\right)-\frac{i k}{2} D_{-a_2}\left(e^{-\frac{\pi i}{4}} k\right)\right], \\
&\mathcal{\hat{M}}^{+}_{21}=e^{\frac{-3 \pi \nu(z_2)}{4}} \beta_{12}^{-1}(z_2)\left[\partial_{k} D_{a_2}\left(e^{-\frac{3 \pi i}{4}} k\right)+\frac{i k}{2} D_{a_2}\left(e^{-\frac{3 \pi i}{4}} k\right)\right], \\
&\mathcal{\hat{M}}^{+}_{22}=e^{\frac{\pi \nu(z_2)}{4} } D_{-a_2}\left(e^{\frac{-i\pi}{4}} k\right).
\end{aligned}
$$
When $k\in \mathbb{C}_{-}$
$$
\begin{aligned}
&\mathcal{\hat{M}}^{-}_{11}=e^{\frac{\pi \nu(z_2)}{4}} D_{a_2}\left(e^{\frac{\pi i }{4}} k\right), \\
&\mathcal{\hat{M}}^{-}_{12}=e^{\frac{-3 \pi \nu(z_2)}{4}} \beta_{21}^{-1}(z_2)\left[\partial_{k} D_{-a_2}\left(e^{\frac{3 \pi i }{4}} k\right)-\frac{i k}{2} D_{-a_2}\left(e^{\frac{3 \pi i }{4}} k\right)\right], \\
&\mathcal{\hat{M}}^{-}_{21}=e^{\frac{\pi \nu(z_2)}{4}} \beta_{12}^{-1}(z_2)\left[\partial_{k} D_{a_2}\left(e^{\frac{\pi i}{4} } k\right)+\frac{i k}{2} D_{a_2}\left(e^{\frac{\pi i}{4}} k\right)\right], \\
&\mathcal{\hat{M}}^{-}_{22}=e^{\frac{-3 \pi \nu(z_2)}{4}} D_{a_2}\left(e^{\frac{3 \pi i}{4} } k\right),
\end{aligned}
$$
where
$$
\begin{aligned}
&\beta_{12}(z_2)=t^{Im\nu(z_2)}\frac{\sqrt{2 \pi} e^{\frac{\pi i}{4}} e^{-\frac{\pi \nu(z_2)}{2}}}{\vartheta(z_2)\Gamma(-i \nu(z_2))},\\
&\beta_{21}(z_2)=-t^{-Im\nu(z_2)}\frac{\sqrt{2 \pi} e^{-\frac{\pi i}{4}} e^{-\frac{\pi \nu(z_2)}{2}}}{\tilde{\vartheta}(z_2) \Gamma(i \nu(z_2))}, ~~~\beta_{12}(z_2)\beta_{21}(z_2)=\nu(z_2),\\
&\vartheta(z_2)=r(z_2)\delta^{-2}_{1}(z_2)e^{iRe\nu(z_2)In(48t\beta z_2+8\alpha t)+8it\beta(z-z_2)^3-16it\beta z_2^3-4it\alpha z_2^2},\\
&\tilde{\vartheta}(z_2)=\kappa\tilde{r}(z_2)\delta^{2}_{1}(z_2)e^{-iRe\nu(z_2)In(48t\beta z_2+8\alpha t)-8it\beta(z-z_2)^3+16it\beta z_2^3-4it\alpha z_2^2},
\end{aligned}
$$
and $D_{a_2}(k)$ is the  solution of
$$
\partial^{2}_{k}D_{a_2}(k)+(\frac{1}{2}-\frac{k^2}{4}+a_2)D_{a_2}(k)=0,~~~a_2=i\nu(z_2).
$$
Therefore, two RHP models are constructed here, let $D^{\varepsilon}_{z_1}$ and $D^{\varepsilon}_{z_2}$ be a disk of radius $\varepsilon$ centered at $z_1$ and $z_2$, so
$$
M_{rhp}^{(2)}(z)= \begin{cases}\tilde{M}_{z_{1}}(k), & z \in D^{\varepsilon}_{z_1},\\ \hat{M}_{z_{2}}(k), & z \in D^{\varepsilon}_{z_2},\end{cases}
$$
consider that $M_{rhp}^{(2)}(z)$ is a Laurent series expansion at $k\rightarrow\infty$
\begin{equation}\label{mrhp}
M_{rhp}^{(2)}(z)=I+\frac{\tilde{M}^{(1)}_{z_{1}}}{k}+\frac{\hat{M}^{(1)}_{z_{2}}}{k}+\mathcal{O}(k^{-2}),
\end{equation}
and $(\tilde{M}^{(1)}_{z_{1}})_{12}=-i\beta_{12}(z_1),(\hat{M}^{(1)}_{z_{1}})_{12}=-i\beta_{12}(z_2)$.

\subsection{Asymptotic analysis on the pure $\bar{\partial}$-problem}
This part mainly analyzes the pure $\bar{\partial}$-problem about $E(z)$, where $E(z)$ satisfies the following RHP:
\begin{rhp}
Find a matrix value function $E(z)$ admitting:

$\bullet$ Analyticity: $E(z)$  is continuous in $\mathbb{C}$ and  continuous in $\mathbb{C}\backslash {{\Sigma}}$ with first partial derivatives;

$\bullet$ Asymptotic behavior: $E(z)=I+\mathcal{O}\left(z^{-1}\right), \quad z \rightarrow \infty;$

$\bullet$  $\bar{\partial}$ derivative: $\bar{\partial}{E}(z)=E(z)M_{rhp}^{(2)} \bar{\partial} \mathcal{R}^{(1)}(z)\left(M_{rhp}^{(2)}\right)^{-1}$.
\end{rhp}

The above RHP is equivalent to the following integral equation
\begin{equation}\label{m3}
E(z)=I-\frac{1}{\pi} \iint_{\mathbb{C}} \frac{E W}{s-z} \mathrm{~d} A(s),
\end{equation}
where $W(z)=M^{(2)}_{rhp}(z) \bar{\partial} R^{(1)} M^{(2)}_{rhp}(z)^{-1},$ and $dA(s)$ is the Lebesgue measure on the real plane. In fact, Eq.(\ref{m3}) can also be written in the form of an operator
\begin{equation}\label{ge}
(I-\mathcal{G}) E(z)=I,
\end{equation}
where $\mathcal{G}$  is Cauchy-Green operator,
$$
\mathcal{G}[f](z)=-\frac{1}{\pi} \iint_{\mathbb{C}} \frac{f(s) W(s)}{s-z} \mathrm{~d} A(s).
$$
If the operator $(I-\mathcal{G})^{-1}$ exists, then the above equation has a solution. That is to say, we need to prove that when t is large enough, the operator $\mathcal{G}$ is small enough to deduce the existence of $(I-\mathcal{G})^{-1}$. We give the following proposition:
\begin{proposition}
For sufficiently large $t$, operator $\mathcal{G}$  is a small norm and has
\begin{equation}\label{sl}
\|\mathcal{G}\|_{L^{\infty} \rightarrow L^{\infty}}\leqslant
\left\{\begin{array}{l}
c t^{-\frac{1}{4}+\frac{1}{2}\max \{\operatorname{Im} \nu(z_{1}), \operatorname{Im}\nu(z_{2})\}}, \quad 0<\operatorname{Im}\nu(z_{i})<\frac{1}{2} \\
c t^{-\frac{1}{4}}, \quad -\frac{1}{2}<\operatorname{Im}\nu(z_{i})\leq 0, \quad i=0,1 \\
\end{array}\right.
\end{equation}
therefore, $(I-\mathcal{G})^{-1}$ exists.
\end{proposition}
\begin{proof}
This is discussed in detail in the $\Omega_4$ region, and other regions can be obtained similarly. Let $s=m+z_2+in,z=\zeta+i\xi$, and $|m|>|n|$, for any $f \in L^{\infty}(\Omega_4)$, we have
\begin{equation}
\begin{aligned}
|\mathcal{G}(f)| &\leq \frac{1}{\pi} \iint_{\Omega_4} \frac{\left|f M_{rhp}^{(2)} \bar{\partial} \mathcal{R}^{(1)} M_{rhp}^{(2)}{ }^{-1}\right|}{|s-z|} d A(s)\\
&\leq \frac{1}{\pi}\|f\|_{L^{\infty}}\left\|M_{rhp}^{(2)}\right\|_{L^{\infty}}\left\|M_{rhp}^{(2)}{ }^{-1}\right\|_{L^{\infty}} \iint_{\Omega_4}\frac{\bar{\partial}\mathcal{R}^{(1)}}{|s-z|} dA(s).
\end{aligned}
\end{equation}
Reviewing $Re(2it\theta)$ to know
$$
\begin{aligned}
Re(2 i t \theta) &=8 \beta t(\left(-3 m^{2} n+ n^3\right)-4t(12 \beta z_{2}+2 \alpha) m n \\
&\lesssim 8 \beta t(\left(-3 m^{2} v+m^{2} n\right)-4t(12 \beta z_{2}+2 \alpha) m n \\
& \lesssim -16 \beta t  m^{2} n-4t(12 \beta z_{2}+2 \alpha) m n \\
& \lesssim-16 \beta t  m n.
\end{aligned}
$$
Combined with the boundedness of $\|M_{rhp}^{(2)}\|_{L^{\infty}}$ and Eq.(\ref{rj}), the above integral can be reduced to
\begin{equation}\label{gf}
\begin{aligned}
|\mathcal{G}(f)|
&\leq c \int_{0}^{\infty} \int_{n}^{\infty}\frac{\left|\bar{\partial} R_{4}\right| e^{-16 \beta t  m n}}{|s-z|} dm dn \leq c\left(\hbar_{1}+\hbar_{2}\right),
\end{aligned}
\end{equation}
where
$$
\begin{aligned}
&\hbar_{1}=\int_{0}^{\infty} \int_{n}^{\infty}\frac{\left|r^{\prime}(\operatorname{Re} s)\right|e^{-16 \beta t  m n}}{|s-z|}dm dn,\\
&\hbar_{2}=\int_{0}^{\infty} \int_{n}^{\infty}\frac{|s-z_2|^{-\frac{1}{2}-Im\nu(z_2)}e^{-16 \beta t  m n}}{|s-z|}dm dn.
\end{aligned}
$$
Theorem \ref{rh} shows that $r(z),\tilde{r}(z)\in H^{1,1}(\mathbb{R})$, then $r^{\prime}(z),\tilde{r}^{\prime}(z)\in L^2(\mathbb{R})$, so
$$
\begin{aligned}
\hbar_{1}&=\int_{0}^{\infty} \int_{n}^{\infty}\frac{\left|r^{\prime}(\operatorname{Re} s)\right|e^{-16 \beta t  m n}}{|s-z|}dm dn\\
&\leq \left(\left\|r^{\prime}\right\|_{L^{2}}\right) \int_{0}^{\infty} e^{-t n^{2}} \mathrm{~d} v\left(\int_{n}^{\infty}|s-z|^{-2} \mathrm{~d} m\right)^{\frac{1}{2}},
\end{aligned}
$$
besides,
\begin{equation}\label{sk}
\begin{aligned}
\left\|\frac{1}{s-z}\right\|_{L^{2}\left(n, \infty\right)}^{2}&=\int_{n}^{\infty} \frac{1}{|s-z|^{2}} d m \leq \int_{-\infty}^{\infty} \frac{1}{|s-z|^{2}} d m \\
&=\int_{-\infty}^{\infty} \frac{1}{(m+z_2-\zeta)^{2}+(n-\xi)^{2}} d m=\frac{1}{|n-\xi|} \int_{-\infty}^{\infty} \frac{1}{1+x^{2}} d x=\frac{\pi}{|n-\xi|},
\end{aligned}
\end{equation}
where $x=\frac{m+z_2-\zeta}{n-\xi}$.
Therefore, for $\hbar_{1}$
\begin{equation}\label{h1}
\hbar_{1}\lesssim \left\|r^{\prime}\right\|_{L^{2}} \int_{0}^{\infty} \frac{e^{-t n^{2}}}{|n-\xi|^{\frac{1}{2}}} \mathrm{~d} n \lesssim c t^{-\frac{1}{4}}.
\end{equation}
For the estimation of $\hbar_{2}$, we need the help of H\"{o}lder inequality with $ p> 2$ and $\frac{1}{p}+\frac{1}{q}=1$.
\begin{equation}\label{slp}
\begin{aligned}
\left\||s-z_2|^{-\frac{1}{2}-Im\nu(z_2)}\right\|_{L^{p}(n,\infty)}&=\left(\int_{n}^{\infty} {|m+i n|^{-\frac{p}{2}-Im\nu(z_2)p }} d m\right)^{\frac{1}{p} }=\left(\int_{n}^{\infty} {\left(m^{2}+n^{2}\right)^{-\frac{p}{4}-\frac{Im\nu(z_2)p}{2}}}d m\right)^{\frac{1}{p}} \\
&=n^{\frac{1}{p}-\frac{1}{2}-Im\nu(z_2)}\left(\int_{1}^{\infty} {\left(1+x^{2}\right)^{-\frac{p}{4}-\frac{Im\nu(z_2)p}{2}}} d x\right)^{\frac{1}{p}} \leq c n^{\frac{1}{p}-\frac{1}{2}-Im\nu(z_2)}.
\end{aligned}
\end{equation}
Similar estimates can be proved
$$
\left\|\frac{1}{s-z}\right\|_{L^{q}(n, \infty)} \leq c|n-\xi|^{\frac{1}{q}-1}, \quad \frac{1}{q}+\frac{1}{p}=1.
$$
Then we can easily prove that
\begin{equation}\label{h2}
\begin{aligned}
\hbar_{2} &\leq c \int_{0}^{+\infty}\left\||s-z_2|^{-\frac{1}{2}-Im\nu(z_2)}\right\|_{L^{p}}\left\|\frac{1}{s-z}\right\|_{L^{q}}e^{-16 \beta t  m n}d n\\
&\leq c\left[\int_{0}^{\xi}  n^{\frac{1}{p}-\frac{1}{2}-Im\nu(z_2)}|n-\xi|^{\frac{1}{q}-1}e^{-n^2t} d n+\int_{\xi}^{\infty} n^{\frac{1}{p}-\frac{1}{2}-Im\nu(z_2)}|n-\xi|^{\frac{1}{q}-1}e^{-n^2t} d n\right]\\
&=c \int_{0}^{1} \xi^{\frac{1}{2}-Im\nu(z_2)}\eta^{\frac{1}{p}-\frac{1}{2}-Im\nu(z_2)}(1-\eta)^{\frac{1}{q}-1} e^{-\eta^2\xi^2t}d \eta+\int_{0}^{\infty} (\eta+\xi)^{\frac{1}{p}-\frac{1}{2}-Im\nu(z_2)}\eta^{\frac{1}{q}-1} e^{-(\eta+\xi)^2t}d \eta\\
&\lesssim \xi^{\frac{1}{2}-Im\nu(z_2)}\int_{0}^{1}\left(t^{\frac{1}{2}}\xi \eta\right)^{-\frac{1}{2}+Im\nu(z_2)} \eta^{\frac{1}{p}-\frac{1}{2}-Im\nu(z_2)}(1-\eta)^{\frac{1}{q}-1} d \eta+\int_{0}^{\infty} e^{- t \eta^{2}} \eta^{-\frac{1}{2}-Im\nu(z_2} d \eta \\
&\lesssim ct^{-\frac{1}{4}+\frac{Im\nu(z_2)}{2}}.
\end{aligned}
\end{equation}
In the $\Omega_1$ region, for any $f \in L^{\infty}(\Omega_1)$, we can get
\begin{equation}\label{gf1}
\begin{aligned}
|\mathcal{G}(f)| &\leq \frac{1}{\pi} \iint_{\Omega_1} \frac{\left|f M_{rhp}^{(2)} \bar{\partial} \mathcal{R}^{(1)} M_{rhp}^{(2)}{ }^{-1}\right|}{|s-z|} d A(s)\\
&\leq \frac{1}{\pi}\|f\|_{L^{\infty}}\left\|M_{rhp}^{(2)}\right\|_{L^{\infty}}\left\|M_{rhp}^{(2)}{ }^{-1}\right\|_{L^{\infty}} \iint_{\Omega_1}\frac{\bar{\partial}{R}^{(1)}}{|s-z|} dA(s)\\
&\leq ct^{-\frac{1}{4}}+ct^{-\frac{1}{4}+\frac{Im\nu(z_1)}{2}}.
\end{aligned}
\end{equation}
The same is true for other areas. With (\ref{h1}),(\ref{h2}) and (\ref{gf1}), the result (\ref{gf}) is verified.
\end{proof}

From the above proposition, we can know the solvability of Eq.(\ref{ge}), and then do the following Laurent series expansion for $E(z)$
$$
E(z)=I+\frac{E^{(1)}}{z}+\mathcal{O}\left(z^{-2}\right), \quad z \rightarrow \infty,
$$
from the previous Eq.(\ref{m3}), it is easy to know
\begin{equation}\label{e1}
E^{(1)}=\frac{1}{\pi} \iint_{\mathbb{C}} E(s) W^{(3)}(s) d A(s).
\end{equation}
Further, we can prove that
\begin{proposition}
For a large t, $E^{(1)}$ meets the following estimate
\begin{equation}\label{m13}
\left|E^{(1)}\right|\leq  \left\{\begin{array}{l}
c t^{-\frac{3}{4}+\frac{1}{2}\max \{\operatorname{Im} \nu(z_{1}), \operatorname{Im}\nu(z_{2})\}}, \quad 0<\operatorname{Im}\nu(z_{j})<\frac{1}{2} \\
c t^{-\frac{3}{4}}, ~~~~~\quad -\frac{1}{2}<\operatorname{Im}\nu(z_{j})\leq 0, \quad j=1,2 \\
\end{array}\right.
\end{equation}
\end{proposition}

\begin{proof}
Here, we still take region $\Omega_4$ as an example to prove that other regions have similar steps. According to Eq.(\ref{m3})
\begin{equation}
\begin{aligned}
\left|E^{(1)}\right| &\leq \frac{1}{\pi} \iint_{\Omega_4} \left|E M_{rhp}^{(2)} \bar{\partial} \mathcal{R}^{(1)} M_{rhp}^{(2)}{ }^{-1}\right| d A(s)\\
&\leq \frac{1}{\pi}\left\|E\right\|_{L^{\infty}}\left\|M_{rhp}^{(2)}\right\|_{L^{\infty}}\left\|\left(M_{rhp}^{(2)}\right)^{-1}\right\|_{L^{\infty}} \iint_{\Omega_4}\left|\bar{\partial} R_{4} e^{2 i t \theta}\right| d A(s),
\end{aligned}
\end{equation}
the boundedness of $\|E\|_{L^{\infty}},\|M_{rhp}^{(2)}\|_{L^{\infty}}$ and the estimation of $\bar{\partial} R_{4}$ (\ref{rj}), the above estimation be reduced to
\begin{equation}
\begin{aligned}
\left|E^{(1)}\right|&\leq c \iint_{\Omega_4}\left(|r^{\prime}(\operatorname{Re} s)|+|s-z_2|^{-\frac{1}{2}-Im\nu(z_2)}\right)e^{-16 \beta t  m n} d A(s)\\
&\leq c\left(\ell_1+\ell_2\right)
\end{aligned}
\end{equation}
where
$$
\begin{aligned}
&\ell_1=\iint_{\Omega_4}|r^{\prime}(\operatorname{Re} s)|e^{-16 \beta t  m n}dA(s),\\
& \ell_2=\iint_{\Omega_4}|s-z_2|^{-\frac{1}{2}-Im\nu(z_2)}e^{-16 \beta t  m n} d A(s).
\end{aligned}
$$
We constrain $\ell_{1}$ by using the Cauchy-Schwarz inequality
$$
\begin{aligned}
\ell_{1} &=\int_{0}^{\infty} \int_{v}^{\infty}\left|r^{\prime}(\operatorname{Re} s)\right| e^{- t  m n} \mathrm{d} m \mathrm{~d} n \\
& \leq \int_{0}^{\infty} e^{- t n^{2}}\left\|r^{\prime}\right\|_{L^{2}}\left(\int_{0}^{\infty} e^{-t  m n} \mathrm{~d} m\right)^{\frac{1}{2}} \mathrm{~d} n \\
& \leq c_3 t^{-\frac{1}{2}}\int_{0}^{\infty} \frac{e^{- t n^{2}}}{ \sqrt{t n}} \mathrm{~d} n \lesssim c t^{-\frac{3}{4}}.
\end{aligned}
$$
For the constraint of $\ell_2$, we follow the method of $\hbar_{2}$ and use H\"{o}lder  inequality and Eq.(\ref{slp}) to obtain
$$
\begin{aligned}
\ell_{2}=& \int_{0}^{+\infty} \int_{n}^{+\infty}\left(m^{2}+n^{2}\right)^{-\frac{1}{4}-\frac{Im\nu(z_2)}{2}} e^{-16 \beta t  m n}  d m d n \\
& \leq \int_{0}^{+\infty}\left\|\left(m^{2}+n^{2}\right)^{-\frac{1}{4}-\frac{Im\nu(z_2)}{2}}\right\|_{L^{p}}\left(\int_{n}^{+\infty} e^{- qt  m n}du \right)^{\frac{1}{q}} d n\\
& \leq \int_{0}^{+\infty}n^{\frac{1}{p}-\frac{1}{2}-Im\nu(z_2)}\left(\int_{n}^{+\infty} e^{- qt  m n}dm \right)^{\frac{1}{q}} d n\\
&\leq c t^{-\frac{1}{q} } \int_{0}^{+\infty}n^{\frac{2}{p}-\frac{3}{2}-Im\nu(z_2)} e^{- t  n^2} d n\\
&\leq c t^{-\frac{3}{4}+\frac{Im\nu(z_2)}{2} } \int_{0}^{\infty} \eta^{\frac{2}{p} -\frac{3}{2}-Im\nu(z_2) } e^{-\eta^2 t} d \eta\\
 &\leq c t^{-\frac{3}{4}+\frac{Im\nu(z_2)}{2}}.
\end{aligned}
$$
It is worth noting that $2<p<4$ and $-1<\frac{2}{p} -\frac{3}{2}<-\frac{1}{2}$ guarantee the convergence of generalized integrals. Similarly, in region $\Omega_1$, we can get
\begin{equation}
\begin{aligned}
\left|E^{(1)}\right| &\leq \frac{1}{\pi} \iint_{\Omega_4} \left|E M_{rhp}^{(2)} \bar{\partial} \mathcal{R}^{(1)} M_{rhp}^{(2)}{ }^{-1}\right| d A(s)\\
&\leq \frac{1}{\pi}\left\|E\right\|_{L^{\infty}}\left\|M_{rhp}^{(2)}\right\|_{L^{\infty}}\left\|\left(M_{rhp}^{(2)}\right)^{-1}\right\|_{L^{\infty}} \iint_{\Omega_1}\left|\bar{\partial} R_{1} e^{2 i t \theta}\right| d A(s)\\
&\leq ct^{-\frac{3}{4}}+c t^{-\frac{3}{4}+\frac{Im\nu(z_1)}{2}}.
\end{aligned}
\end{equation}
Combined with the above estimates, the proposition is proved.
\end{proof}

\section{Long time asymptotic behavior of the solution for the nonlocal Hirota equation}
This section mainly combines the above transformations to derive the Long time asymptotic behavior of the solution for the nonlocal Hirota equation. After the transformation in section 5 above, the following formula can be obtained
\begin{equation}\label{mt}
M(z)=E(z) M_{rhp}^{(2)}(z) \mathcal{R}^{(1)}(z)^{-1} \delta(z)^{\sigma_{3}},
\end{equation}

In particular, in the vertical direction $z\in \Omega_2,\Omega_{6}$, there is $\mathcal{R}^{(1)} = I$, so we consider the Laurent series of $M$ in these two regions
$$
M(z)=\left(I+\frac{E^{(1)}}{z}+\ldots\right)\left(I+\frac{{M_{rhp}^{(2)}}^{(1)}}{z}+\ldots\right)\left(I+\frac{{\delta^{(1)}}^{\sigma_{3}}}{z}+\ldots\right).
$$
Using Eqs. (\ref{dz}), (\ref{mrhp}) and (\ref{e1}), we can obtain the asymptotic behavior of $M$
$$
M(z)= I+\left(E^{(1)}+\frac{\tilde{M}_{z_1}^{(1)}}{\sqrt{8\left(6 \beta z_{1}+\alpha\right) t}}+\frac{\hat{M}_{z_2}^{(1)}}{\sqrt{8\left(6 \beta z_{2}+\alpha\right) t}}-\left(i \int_{z_1}^{z_2} \nu(s) \mathrm{d} s\right)^{\sigma_{3}}\right) z^{-1}+\mathcal{O}\left(z^{-2}\right).
$$
Finally, the long-time asymptotic behavior of nonlocal Hirota equation is expressed as
\begin{theorem}
Let $q(x,t)$ be the solution of the nonlocal Hirota equation with initial value belongs to $H^{1,1}(\mathbb{R})$. As $t\rightarrow\infty$, the leading asymptotics of the solution can be expressed as
$$
\begin{aligned}
q(x,t)
&=2\left(iE^{(1)}_{12}+\frac{\beta_{12}(z_1)}{\sqrt{8\left(6 \beta z_{1}+\alpha\right) t}}+\frac{\beta_{12}(z_2)}{\sqrt{8\left(6 \beta z_{2}+\alpha\right) t}}\right)\\
&=\frac{t^{-\frac{1}{2}+Im\nu(z_1)}\sqrt{\pi} e^{\frac{\pi i}{4}} e^{-\frac{\pi \nu(z_1)}{2}}}{\sqrt{\left(6 \beta z_{1}+\alpha\right) }\vartheta(z_1)\Gamma(-i \nu(z_1))}+\frac{t^{-\frac{1}{2}+Im\nu(z_2)}\sqrt{\pi} e^{\frac{\pi i}{4}} e^{-\frac{\pi \nu(z_2)}{2}}}{\sqrt{\left(6 \beta z_{2}+\alpha\right) }\vartheta(z_2)\Gamma(-i \nu(z_2))}+\Xi(t),\\
&\Xi(t)=\left\{\begin{array}{l}
 \mathcal{O}(t^{-\frac{3}{4}+\frac{1}{2}\max \{\operatorname{Im} \nu(z_{1}), \operatorname{Im}\nu(z_{2})\}}), \quad 0<\operatorname{Im}\nu(z_{j})<\frac{1}{2} \\
\mathcal{O}( t^{-\frac{3}{4}}), \quad -\frac{1}{2}<\operatorname{Im}\nu(z_{j})\leq 0 \\
\end{array}\right.\\
&\vartheta(z_j)=r(z_j)\delta^{-2}_{j}(z_j)e^{iRe\nu(z_j)In(48t\beta z_j+8\alpha t)+8it\beta(z-z_j)^3-16it\beta z_j^3-4it\alpha z_j^2},\\
&\nu(z_j)=-\frac{1}{2 \pi} \log \left(1-\kappa r(z_j)\tilde{r}(z_j)\right),~~j=1,2.\\
\end{aligned}
$$
\end{theorem}

\section*{Acknowledgements}
The project is supported by the the National Natural Science Foundation of China (No. 12175069) and Science and Technology Commission of Shanghai Municipality (No. 21JC1402500 and No. 18dz2271000).





\end{document}